\documentclass[a4paper,11pt]{article}

\usepackage{cmap}

\usepackage[usenames,dvipsnames]{xcolor}
\definecolor{myblue}{rgb}{0.21, 0.34, 0.74}
\definecolor{mygrey}{rgb}{0.55, 0.57, 0.67}
\definecolor{myred}{rgb}{0.79, 0.0, 0.09}

\definecolor{lavender}{HTML}{CEB6EB}
\definecolor{softblue}{HTML}{95C1FF}
\definecolor{softgreen}{HTML}{BBE2B1}

\usepackage[pdftex, colorlinks=true, bookmarksopen=true, linkcolor=blue, citecolor=blue]{hyperref}

\usepackage[T1]{fontenc}
\usepackage[utf8]{inputenc}

\usepackage{lmodern}
\usepackage{libertine}

\usepackage{tikz-cd}

\usepackage{enumitem}
\usepackage{bbm}
\usepackage{mathrsfs}
\usepackage{amssymb}
\usepackage{amsmath}
\usepackage{upgreek}
\usepackage{mathtools}

\usepackage{amsthm}
\usepackage[nottoc,notlot,notlof]{tocbibind}

\usepackage[noabbrev,capitalize,nameinlink]{cleveref}

\usepackage{comment}
\usepackage[font=small,labelfont=bf]{caption}

\theoremstyle{plain}
\newtheorem{theorem}{Theorem}[section]
\newtheorem{definition}[theorem]{Definition}
\newtheorem{proposition}[theorem]{Proposition}
\newtheorem{lemma}[theorem]{Lemma}

\newtheorem{claim}{Claim}
\newtheorem{remark}[theorem]{Remark}

\usepackage{authblk}

\usepackage{subcaption}

\usepackage{tikz}
\usetikzlibrary{calc,intersections}

\usepackage[svgnames,pdf]{pstricks}

\usepackage{wrapfig}
\usepackage{upgreek}
\usepackage{bm}
\usepackage[scr=boondoxo]{mathalfa}
\usepackage{comment}
\usepackage[font=small,labelfont=bf]{caption}

\newcommand{\supp}{\mathrm{supp}}

\newcommand{\R}{\mathbb{R}}

\newcommand{\diff}{\, \mathrm{d}}
\newcommand{\dist}{\mathrm{dist}}

\newcommand{\N}{\mathbb{N}}

\usepackage{esint}  % in preamble

\newcommand{\Id}{\mathsf{id}}

\mathchardef\mhyphen="2D % Define a "math hyphen"

\renewcommand{\leq}{\leqslant}
\renewcommand{\geq}{\geqslant}
\renewcommand{\le}{\leqslant}
\renewcommand{\ge}{\geqslant}

\renewcommand{\rho}{\uprho}

\renewcommand{\mu}{\upmu}

\renewcommand{\nu}{\upnu}

\renewcommand{\beta}{\upbeta}

\renewcommand{\psi}{\uppsi}
\renewcommand{\phi}{\upphi}
\renewcommand{\varphi}{\upvarphi}
\renewcommand{\omega}{\upomega}
\renewcommand{\eta}{\upeta}
\renewcommand{\delta}{\updelta}

\numberwithin{equation}{section}

\title{ Constructive conditional normalizing flows  }

\author{Borjan Geshkovski}
\affil{Laboratoire Jacques-Louis Lions\\ Inria \& Sorbonne Université}
\author{Domènec Ruiz-Balet}
\affil{Universitat de Barcelona}

\date{ \today }

\begin{document}

%------
% Insert the title of your paper and (if necessary)
% a short title for the running head.
%------
\title{ Constructive conditional normalizing flows }

\setlist[itemize,enumerate]{left=0pt}

%
% Title
%

\maketitle

%
% Abstract
%

\begin{abstract}

Motivated by applications in conditional sampling, given a probability measure $\mu$ and a diffeomorphism $\phi$, we consider the problem of simultaneously approximating $\phi$  and the pushforward $\phi_\#\mu$ by means of the flow of a continuity equation whose velocity field is a perceptron neural network with piecewise constant weights. We provide an explicit construction based on a polar-like decomposition of the Lagrange interpolant of $\phi$. The latter involves a compressible component, given by the gradient of a particular convex function, which can be realized exactly, and an incompressible component, which—by after approximating via permutations—can be implemented through shear flows intrinsic to the continuity equation.
For more regular maps $\phi$---such as the Kn\"othe-Rosenblatt rearrangement---, we provide an alternative, probabilistic construction inspired by the Maurey empirical method, in which the number of discontinuities in the weights doesn't scale inversely with the ambient dimension.
			
%\bigskip
%\noindent \textbf{Keywords.}\quad neural ODEs; continuous normalizing flows; optimal transport.

%\medskip
%\noindent \textbf{\textsc{ams} classification.}\quad \textsc{49Q22, 35Q49, 37A05, 37C10, 28D05, 68T07}.
\end{abstract}
	
\thispagestyle{empty}

\setcounter{tocdepth}{3}
\tableofcontents

%------
% INSERT THE BODY OF THE PAPER HERE (except
% acknowledgments, funding info and bibliography)
%------

%
% Intro
%

\section{Introduction}

Consider the controlled ODE
\begin{equation} \label{eq: neural.ode}
\begin{cases}
\dot{x}(t) = v(x(t), \theta(t)), & t\in[0,T],\\
x(0) = x\in\R^d
\end{cases}
\end{equation}
with velocity field given by the \emph{perceptron}\footnote{For simplicity we focus on \emph{width~$1$}. See \Cref{rem: width}.}
\begin{equation} \label{eq: two.layer.mlp}
v(x,\theta) \coloneqq w(a\cdot x+b)_+,
 \hspace{1cm}\theta=(w,a,b)\in\mathbb{R}^{2d+1}.
\end{equation}
Given a control $\theta:[0,T]\to\mathbb{R}^{2d+1}$, we denote by $\phi_\theta^t:\mathbb{R}^d\to\mathbb{R}^d$ the time-$t$ flow map. 
We call a jump discontinuity of~$\theta$ a \emph{switch}. 

One can view \eqref{eq: neural.ode} as the equation of characteristic curves of 
\begin{equation} \label{eq: cont.eq}
\begin{cases}
\partial_t \uprho(t, x) + \nabla_x \cdot (v(x,\theta(t)) \uprho(t,x))=0, & (t,x)\in [0,T]\times\mathbb{R}^d,\\[2pt]
\uprho(0, x) = \uprho_{\mathrm{B}}(x),
\end{cases}
\end{equation}
for a given probability density $\uprho_{\mathrm{B}}$ on $\R^d$.

The continuity equation \eqref{eq: cont.eq}, with the particular choice of velocity field \eqref{eq: two.layer.mlp}, is of central interest in recent applications in machine learning, and goes under the umbrella term of \emph{continuous normalizing flows} \cite{papamakarios2021normalizing, chen2018neural, grathwohl2018ffjord}, whereas \eqref{eq: neural.ode} is called \emph{neural ODE}. In this framework, one typically has access to samples from an unknown density \(\uprho_*\) and chooses a simple \(\uprho_{\mathrm{B}}\) (often a centered Gaussian), which is transported by \eqref{eq: cont.eq} to match \(\uprho_*\) by tuning the controls \(\theta\) so as to minimize a discrepancy (often the relative entropy, due to connections with maximum likelihood estimation). We refer the reader to the introduction in \cite{alvarez2025constructive} for a detailed account. 

In this work, we present a constructive---though not necessarily optimal---approach to solving tasks of this nature. 
Along the way, we look to elucidate some working mechanisms of deep neural networks.

For a probability density $\rho$ on  $\mathbb{R}^d$ we denote the measure by $\mu(A)\coloneqq\int_A \rho(x)\diff x$ for Borel measurable $A$. 
For a Borel measurable map $\psi:\R^d\to\R^d$ we write $\psi_\#\mu$ for the pushforward. For finite measures $\nu$ we use the total variation $|\nu|_{\mathsf{TV}} \coloneqq \sup_{\|f\|_{L^\infty}\le 1}\int f \diff\nu,$ so that $|\mu-\nu|_{\mathsf{TV}}=2\sup_{A}|\mu(A)-\nu(A)|$. Finally, we say that a domain $\Omega\subset\R^d$ is rectangular if it is the union of finitely many bounded orthotopes.

Our first result is the following.

\begin{theorem} \label{thm: gen.thm}
Fix $p\geq1$,  a $C^1$ diffeomorphism $\phi:\R^d\to\R^d$ and a probability density $\rho\in L^\infty\cap C^{0,1}(\R^d)$. 
Then for every rectangular domain $\Omega\subset\R^d$ and  $\varepsilon>0$ there exists a piecewise constant $\theta:[0,T]\to\mathbb{R}^{2d+1}$ with finitely many switches such that 
\begin{equation}\label{eq: Lp.small}
\left\|\phi-\phi_\theta^T\right\|_{L^p(\Omega)}\le \varepsilon
\end{equation}
and
\begin{equation}\label{eq: TV.small}
\left|\phi_\#\mu-\phi_{\theta\#}^T\mu\right|_{\mathsf{TV}(\Omega)}\le \varepsilon.
\end{equation}
\end{theorem}

Ensuring both \eqref{eq: Lp.small} and \eqref{eq: TV.small} has its own intrinsic significance, but it is also motivated by  \emph{conditional sampling}. In many  applications stemming from the physical sciences, one has samples from a probability density $\rho_*:\R^{d_1+d_2}\to\R_{\geq0}$ and wishes to sample from the density $\R^{d_1}\ni x\mapsto\rho^{x|y}_*(x)$ of the conditional distribution $\mu_*(x\,|\,y)$ \cite{ingraham2023illuminating, stuart2010inverse}. In \cite[Theorem 2.4]{baptista2024conditional}, the authors establish that a transportation-of-measure-perspective of this task can be formulated by imposing structure on the transport map. Take $\rho_{\mathrm{B}}=\rho_{\mathrm{B}}^1\otimes\rho_{\mathrm{B}}^2$ with $\rho_{\mathrm{B}}^1$ and $\rho_{\mathrm{B}}^2$ being probability densities on $\R^{d_1}$ and $\R^{d_2}$ respectively.
If $\phi$ is of the form 
\begin{equation} \label{eq: triangular.map}
\phi(x, y) = \begin{bmatrix}
    \phi_1(y)\\
    \phi_2(x,\phi_1(y))
\end{bmatrix}   
\end{equation}
and $\phi_{\#}\mu_{\mathrm{B}}=\mu_*$, then for $\phi_{1\#}\mu_{B}^1$-almost every $y$, it holds $\phi_{2\#}\mu_{B}^2=\mu_*(\cdot\,|y)$. The map $\phi_2$ is referred to as the conditional sampler. 
Triangular transport maps as \eqref{eq: triangular.map} can be shown to exist under minimal assumptions---the canonical example is the Kn\"othe-Rosenblatt rearrangement \cite[Section 1]{villani2009optimal}. This perspective has spurred significant activity in the approximation and construction of such maps \cite{
baptista2024conditional2,%
BaptistaMorrisonZahmMarzouk2024JMLR,%
ParnoMarzouk2018TMAMCMC,%
SpantiniBigoniMarzouk2018JMLR,%
ZechMarzouk2022CAI,%
ZechMarzouk2022CAII,%
BrennanBigoniZahmSpantiniMarzouk2020NeurIPS,%
MarzoukMoselhyParnoSpantini2016HOUQ,%
BaptistaHosseiniNguyenZhang2025KRSoftOT,%
BaptistaHosseiniKovachkiMarzoukSagiv2023ApproxTheory,%
RamgraberSharpLeProvostMarzouk2025FriendlyIntro%
}. 
\Cref{thm: gen.thm} implies that any such transport map $\phi$ can be realized, to arbitrary accuracy, as the time-$T$ flow of \eqref{eq: cont.eq} while retaining control of the joint pushforward. 

The number of switches in \Cref{thm: gen.thm}, while finite, scales at best as $1/\varepsilon^d$ for an arbitrary $\phi$---this is due to the "worst-case" nature of the result. We discuss this further in \Cref{sec: switches}.
In contrast, our second result provides a construction that scales as $1/\varepsilon^2$ for specific diffeomorphisms $\phi$. 

Let $\mathrm{Diff}^1(\R^d)$ be the space of $C^1$ diffeomorphisms, and for $s>d/2+1$ define
\begin{equation*}
    \mathcal{D}^s(\R^d)\coloneqq\left\{\phi\in\mathrm{Diff}^1(\R^d)\colon\phi-\mathsf{id}\in H^s(\R^d;\mathbb{R}^d)\right\}.
\end{equation*}
Denote by $\mathcal D^s_0$ the \emph{connected component of the identity}
\begin{equation*}
\mathcal D^s_0(\R^d)
\coloneqq\left\{\phi^1_u\colon  \partial_t\phi^t=u(t)\circ \phi^t_u,\, \phi^0_u=\mathsf{id},\, u\in L^1([0,1];H^s(\R^d;\R^d))\right\},  
\end{equation*}
which is well-defined by virtue of \cite[Theorem 4.4]{bruveris2017completeness}. All these definitions extend to the setting of bounded Lipschitz domains \cite{bruveris2017completeness}. 

Our second result is the following.  

\begin{theorem} \label{thm:sobolev_to_switches_TV} \label{thm:sobolev_to_switches}
Fix $s>d/2+2$, a diffeomorphism $\phi\in \mathcal{D}^s_0(\R^d)$ with $\mathsf{A}_s(\phi)<\infty$, and a probability density
$\rho\in L^\infty\cap C^{0,1}(\R^d)$ with bounded support. Then for every $\varepsilon>0$ and every bounded domain $\Omega\subset\R^d$ such that $\supp\,\rho\subset\Omega$, there exists a piecewise constant 
$\theta:[0,1]\to\R^{2d+1}$ with at most $N$ switches such that
\begin{equation*}
\|\phi-\phi_\theta^1\|_{L^2(\Omega)}\le \varepsilon,
\end{equation*}
and
\begin{equation*}
\left|\phi_\#\mu-\phi_{\theta\#}^1\mu\right|_{\mathsf{TV}(\Omega)}\le \varepsilon.
\end{equation*}
Moreover one can choose $N$ so that
\[
N\le
\frac{C e^{C\sqrt{\mathsf{A}_s(\phi)}} \mathsf{A}_s(\phi)}{\varepsilon^2},
\]
where $C=C(d,s,\Omega,\rho)>0$ and
\[
\mathsf{A}_s(\phi)
\coloneqq \inf\left\{
\int_0^1 \|u(t,\cdot)\|_{H^s}^2\diff t \colon
u\in L^2([0,1];H^s(\mathbb{R}^d;\mathbb{R}^d)),\ \phi_u^1=\phi
\right\},
\]
with $\phi_u^1$ denoting the time--$1$ flow of the vector field $u(t,\cdot)$.
\end{theorem}

All $\phi\in\mathcal{D}_0^s(\R^d)$ satisfy $\mathsf{A}_s(\phi)<\infty$.
Indeed, by definition
$\phi=\phi_u^1$ for some $u\in L^1([0,1];H^s(\R^d;\R^d))$. After a
reparameterization of time by the $H^s$-arc length of the path, one may assume
that $\|u(t)\|_{H^s}$ is constant for a.e.\ $t\in[0,1]$. In particular,
$u\in L^\infty([0,1];H^s)\subset L^2([0,1];H^s)$.
On the other hand, $\mathcal D_0^s(\R^d)$ is a proper subclass of the
orientation-preserving $C^1$ diffeomorphisms of $\R^d$. The obstruction is
already visible in the definition: the condition
$\phi-\mathsf{id}\in H^s(\R^d;\R^d)$ forces the displacement to decay at
infinity in a Sobolev sense. In particular, a nontrivial translation
$\phi(x)=x+\tau$, with $\tau\in\R^d\setminus\{0\}$, satisfies
$\phi-\mathsf{id}\equiv \tau\notin H^s(\R^d;\R^d)$, so
$\phi\notin \mathcal D_0^s(\R^d)$.

One map that does fit in the setting of the previous theorem is the \emph{Knöthe-Rosenblatt rearrangement}.

\begin{definition} \label{def:KR}
Let $\rho_0,\rho_1$ be probability densities on $(0,1)^d$.
For $k=1,\dots,d$ write $x_{1:k}\coloneqq (x_1,\dots,x_k)$ and define the marginal densities
\[
\rho_i^{1:k}(x_{1:k})
\coloneqq
\int_{(0,1)^{d-k}}\rho_i(x_{1:k},x_{k+1},\dots,x_d)\,\diff x_{k+1}\cdots \diff x_d,
\qquad
\rho_i^{1:0}\equiv 1.
\]
For $k\ge 1$, define the conditional density of $x_k$ given $x_{1:k-1}$ and the associated
conditional distribution function by
\begin{align*}
&\rho_i^{k\,|\,1:k-1}(x_k\,|\,x_{1:k-1})
\coloneqq
\frac{\rho_i^{1:k}(x_{1:k})}{\rho_i^{1:k-1}(x_{1:k-1})},
\\
&F_i^k(t\,|\,x_{1:k-1})
\coloneqq
\int_0^t \rho_i^{k\,|\,1:k-1}(s\,|\,x_{1:k-1})\diff s .
\end{align*}
Assume that for every $k$ and every $x_{1:k-1}\in(0,1)^{k-1}$ the map
$t\mapsto F_i^k(t\,|\,x_{1:k-1})$ is strictly increasing from $(0,1)$ onto $(0,1)$, and denote by
$(F_i^k(\cdot\,|\,x_{1:k-1}))^{-1}$ its inverse in the first variable.
The Kn\"othe--Rosenblatt rearrangement $\phi_{\mathrm{KR}}:(0,1)^d\to(0,1)^d$ pushing $\mu_0$ to $\mu_1$
is the map $\phi_{\mathrm{KR}}=(\phi_1,\dots,\phi_d)$ defined recursively by
\[
\phi_1(x_1)\coloneqq (F_1^1)^{-1}\left(F_0^1(x_1)\right),
\]
and, for $k=2,\dots,d$,
\[
\phi_k(x_{1:k})
\coloneqq
\left(F_1^k(\,\cdot\,|\,\phi_{1:k-1}(x_{1:k-1}))\right)^{-1}
\left(F_0^k(x_k\,|\,x_{1:k-1})\right),
\]
where $\phi_{1:k-1}(x_{1:k-1})\coloneqq(\phi_1(x_1),\dots,\phi_{k-1}(x_{1:k-1}))$.
Then $\phi_{\mathrm{KR}}$ is increasing in $x_k$ for each fixed $x_{1:k-1}$ and satisfies
$(\phi_{\mathrm{KR}})_\#\mu_0=\mu_1$.
\end{definition}

In \Cref{sec: proof.prop.kr} we show the following.

\begin{proposition}\label{prop:KR.in.D}
Let $d\ge 1$ and $s>d/2+1$.

\begin{enumerate}
\item  
Suppose $\rho_0,\rho_1\in C^{s}([0,1]^d)$ are positive on $[0,1]^d$.
Let $\phi_{\mathrm{KR}}:(0,1)^d\to (0,1)^d$ be the
Knöthe--Rosenblatt rearrangement pushing $\mu_0$ to $\mu_1$.
Then $\phi_{\mathrm{KR}}$ is a $C^{s}$ diffeomorphism satisfying $\phi_{\mathrm{KR}}\in \mathcal D_0^s((0, 1)^d)$ and $\mathsf{A}_s(\phi_{\mathrm{KR}})<\infty$.

\item Let $\phi: (0,1)^d\to (0,1)^d$ be a triangular map of class $C^{s}$:
\[
\phi(x_1,\dots,x_d)=\begin{bmatrix}
 \phi_1(x_1)\\
 \phi_2(x_1,x_2)\\
 \vdots\\
 \phi_d(x_1,\dots,x_d)
\end{bmatrix}
\]
such that for every $k$ and every $x_{1:k-1}\in [0,1]^{k-1} $, the function
$x_k\mapsto \phi_k(x_{1:k-1},x_k)$ is strictly increasing and maps $(0,1)$ onto $(0,1)$.
Then $\phi$ is a $C^{s}$ diffeomorphism satisfying $\phi\in \mathcal D_0^s((0,1)^d)$ and $\mathsf{A}_s(\phi)<\infty$.
\end{enumerate}
\end{proposition}

The assumption that $x_k\mapsto \phi_k(x_{1:k-1},x_k)$ is strictly increasing ensures that the map is orientation-preserving.

\subsection{Overview of the proofs and outline}

The proofs of \Cref{thm: gen.thm} and \Cref{thm:sobolev_to_switches} rely on different ideas. 

\Cref{thm: gen.thm} is proved by an explicit geometric construction.
We first approximate the target diffeomorphism $\phi$ (along with $\phi^{-1}$, $\nabla\phi$ and $\nabla\phi^{-1}$) on $\Omega$ by a  Lagrange interpolant $\phi_\varepsilon$ (\Cref{thm: lagrange.approx}).
A key step is then a polar-like factorization of $\phi_\varepsilon$ (\Cref{prop: lagrange.decomposition}) into three elementary components:
two measure-preserving maps $m_{1\varepsilon},m_{2\varepsilon}$ and a simple compressible map $g_\varepsilon=\nabla\varphi$ that acts only on one coordinate and is piecewise affine and monotone.
The compressible factor $g_\varepsilon$ can be implemented exactly by a flow of \eqref{eq: neural.ode} with piecewise-constant parameters (\Cref{lem: exact.linear.nonuniform}).
The measure-preserving factors are handled by approximating an arbitrary measure-preserving map by a permutation of small cubes, and then realizing that permutation as a finite composition of explicit divergence-free ``swap'' flows (\Cref{lem: univ.approx.flow} and \Cref{lem: mp.neural.ode.swap}), very much inspired by the work of Brenier and Gangbo \cite{brenier2003approximation}.
Composing these three realizations yields a controlled flow $\phi_\theta^T$ that approximates $\phi$ in $L^p(\Omega)$, and the stability estimates in \Cref{sec: one.other.ineq}--\Cref{thm: stability.measure.preserving} propagate this approximation to the pushforward measures.
This proof is thus fully constructive. See \Cref{sec: proof.th1} (and specifically \Cref{sec: proof.1}).

\Cref{thm:sobolev_to_switches} is of probabilistic nature.
We write $\phi$ as the flow map of some $u\in L^2([0,1];H^s)$, and Sobolev regularity implies that $u(t,\cdot)$ is in the Barron class \cite{barron2002universal, wiener1932tauberian} for a.e. $t$.
We then employ an avatar of Maurey's empirical method \cite{pisier1981remarques}---here essentially a Monte--Carlo argument---we partition time into $N$ intervals and, on each interval, sample a single neuron from the representing Barron measure and freeze it on that interval, producing a piecewise-constant control with exactly $N$ switches. This  can also be interpreted as an avatar of  \emph{chattering control}: the Barron representation of $u(t,\cdot)$ acts as a relaxed control, namely a Young measure on parameters. 
A variance estimate combined with a discrete Gr\"onwall argument yields an error of order $N^{-\frac12}$. See \Cref{sec: proof.2}.

\subsection{Discussion}

\subsubsection{Does one inequality imply the other?} \label{sec: one.other.ineq}

$L^p$–approximation of maps does not, by itself, control the total variation distance between pushforward measures. We first illustrate this through examples.

\begin{enumerate}
    \item On $\Omega=\mathbb{T}^d$ let $h=1/N$ and round each coordinate down to the nearest gridpoint:
    \[
      Q_h(x)\coloneqq h\left\lfloor \frac{x}{h}\right\rfloor\qquad(x\in\Omega).
    \]
    Then $\|Q_h-\mathsf{id}\|_{L^p(\Omega)}\lesssim h$, but for any absolutely continuous $\mu$,
    \(
      Q_{h\#}\mu
    \)
    is purely atomic, whereas $\mathsf{id}_\#\mu=\mu$; hence
    \(
      \left|\mathsf{id}_\#\mu-Q_{h\#}\mu\right|_{\mathsf{TV}(\Omega)}=2.
    \)

    \item On the circle $\mathbb{T}^1$ define, for $0<\alpha<1/(2\pi)$ and $h\in(0,1)$,
\[
  \psi_h(x)\coloneqq x+\alpha h\sin\left(\frac{2\pi x}{h}\right)\quad(\mathrm{mod}\ 1).
\]
Then $\|\psi_h-\mathsf{id}\|_{L^\infty(\mathbb{T}^d)}\le \alpha h$. Notice that $\psi_h'(x)=1+ 2\pi\alpha \cos({2\pi x}/{h})$, which oscillates between $1-2\pi\alpha$ and $1+2\pi\alpha$ (hence stays positive, so $\psi_h$ is a $C^1$ diffeomorphism).
Taking $\mu$ to be the Lebesgue measure, the pushforward density is $\psi_{h\#}\mu = (\psi_h^{-1})'\diff x$, and by the change–of–variables formula
\[
  |\mu-\psi_{h\#}\mu|_{\mathsf{TV}(\mathbb{T}^1)}
  = \int_{\mathbb{T}^1}\left|\psi_h'(x)-1\right| \mathrm{d}x
  = 2\pi\alpha \int_0^1 \left|\cos\left(\frac{2\pi x}{h}\right)\right|\mathrm{d}x.
\]
So we get $|\mu-\psi_{h\#}\mu|_{\mathsf{TV}}\to4\alpha$ as $h\to0$  even though $\|\psi_h-\mathsf{id}\|_{L^\infty}\to 0$.
In particular, this shows that previous results on approximation of the flow maps via \eqref{eq: neural.ode}  \cite{li2022deep,ruiz2023neural,cheng2025interpolation}, as well as \cite{brenier2003approximation}, are insufficient for our purposes.
    
    \item Given $\mu_0,\mu_1\in\mathscr{P}_{\mathrm{ac}}(\Omega)$, the Brenier map $\phi_{\mathrm{OT}}$ and the Kn\"othe–Rosenblatt map $\phi_{\mathrm{KR}}$ both push $\mu_0$ to $\mu_1$, hence
    $|\phi_{\mathrm{OT}\#}\mu_0-\phi_{\mathrm{KR}\#}\mu_0|_{\mathsf{TV}(\Omega)}=0$. However $\|\phi_{\mathrm{OT}}-\phi_{\mathrm{KR}}\|_{L^p(\Omega)}$ can be large. In that regard, also \cite{ruiz2024control} does not answer the question of approximating the flow-map.
\end{enumerate}
We discuss the fundamental reason behind these examples.
For a $C^1$ diffeomorphism $\psi$ and an absolutely continuous $\mu$ with density $\rho$,
\[
(\psi_\#\mu)(B)=\int_{\psi^{-1}(B)} \rho(x)\mathrm{d}x
\quad\Longrightarrow\quad
\frac{\mathrm{d}(\psi_\#\mu)}{\mathrm{d}x}(y)
=\frac{\rho(\psi^{-1}(y))}{|\det\nabla\psi(\psi^{-1}(y))|}.
\]
Let $\phi,\psi:\Omega\to\Omega$ be $C^1$ diffeomorphisms  and set $\eta\coloneqq\phi^{-1}\circ\psi$. Since total variation is invariant under pushforward by a bijection,
\[
|\phi_\#\mu-\psi_\#\mu|_{\mathsf{TV}}
=|(\phi^{-1})_\#(\phi_\#\mu-\psi_\#\mu)|_{\mathsf{TV}}
=|\mu-\eta_\#\mu|_{\mathsf{TV}}.
\]
Using the formula above with $\eta$ gives the explicit identity
\[
|\phi_\#\mu-\psi_\#\mu|_{\mathsf{TV}}
=\int_{\Omega}|\rho(\eta(x))|\det\nabla\eta(x)|-\rho(x)|\diff x.
\]
Thus $\mathsf{TV}$ responds to the deviation of $|\det\nabla\eta|$ from $1$, which is the effect of the \emph{compressible part} of the maps---one may recall Brenier's polar factorization theorem, \cite{brenier1991polar}. 
For general non‑injective maps there is also a multiplicity effect: the pushforward density is a sum over preimages, which can further increase $\mathsf{TV}$ even if pointwise displacements are small. By contrast, $\|\psi-\phi\|_{L^p}$ only measures displacement and is insensitive to Jacobian and multiplicity. Indeed, one can separate displacement and Jacobian effects:
\[
|\rho(\eta)|\det\nabla\eta|- \rho|
\le |\rho(\eta)-\rho| + \|\rho\|_{L^\infty}\,||\det\nabla\eta|-1|,
\]
hence
\[
|\phi_\#\mu-\psi_\#\mu|_{\mathsf{TV}}
\le \|\rho\|_{C^{0,1}}\|\eta-\mathsf{id}\|_{L^1(\Omega)}
+\|\rho\|_{L^\infty}\||\det\nabla\eta|-1\|_{L^1(\Omega)}.
\]
In the measure‑preserving class, this obstruction disappears and small map displacement controls $\mathsf{TV}$. Indeed if both $\phi$ and $\psi$ are measure preserving, then $|\det\nabla\eta|=1$, and if moreover $\phi^{-1}$ is $L$–Lipschitz, then
\[
\|\eta-\mathsf{id}\|_{L^1(\Omega)}
=\|\phi^{-1}\circ\psi-\phi^{-1}\circ\phi\|_{L^1(\Omega)}
\le L\,\|\psi-\phi\|_{L^1(\Omega)}.
\]
This shows the fundamental obstruction is indeed the compressible part.

\subsubsection{On the number of switches} \label{sec: switches}

The number of switches in our construction is dependent solely on the approximation of arbitrary measure-preserving maps by ones which stem from solutions of \eqref{eq: neural.ode}, and is, at best, of order $1/\varepsilon^{d}$. We discuss the difficulties of having better estimates in the worst case in the following.

Let $\mathscr{G}$ be a finite group and $S\subset \mathscr{G}$ a symmetric generating set ($S=S^{-1}$).
The \emph{Cayley graph} $\mathrm{Cay}(\mathscr{G},S)$ has vertex set $\mathscr{G}$ and undirected edges
$\{g,sg\}$ for $g\in \mathscr{G}$, $s\in S$. 
The associated word metric and the diameter of the graph are defined as
\[
\mathrm{dist}_S(e,g)=\min\{\ell: g=s_1\cdots s_\ell, s_i\in S\}, \hspace{0.1cm}
\mathrm{diam}(\mathscr{G},S)=\max_{g\in \mathscr{G}}\mathrm{dist}_S(e,g).
\]

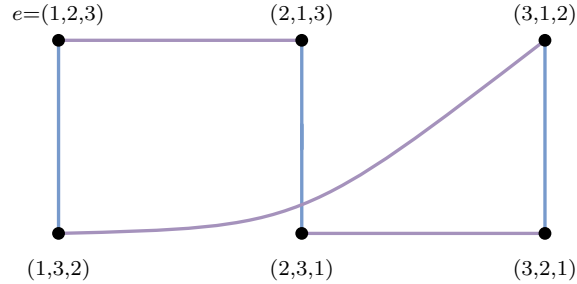
\begin{figure}[h!]
\centering
\begin{tikzpicture}[line cap=round, line join=round, scale=0.8]

  % ---------------- adjustable parameters ----------------
  \def\lw{1.25pt}        % line width for edges
  \def\dotrad{3pt}    % vertex dot radius
  \def\overlapL{0.55cm} % dashed-zone extent to the left of the crossing (along the diagonal)
  \def\overlapR{0.55cm} % dashed-zone extent to the right of the crossing (along the diagonal)
  \def\bluetint{45}     % 0..100, how light the blue is in the overlap zone (bigger => lighter)

  % ---------------- coordinates: three columns, two rows ----------------
  \coordinate (TL) at (0,  1.6);
  \coordinate (TM) at (4,  1.6);
  \coordinate (TR) at (8,  1.6);
  \coordinate (BL) at (0, -1.6);
  \coordinate (BM) at (4, -1.6);
  \coordinate (BR) at (8, -1.6);

  % ---------------- compute intersection and cut points ----------------
  % name (undrawn) paths for the middle vertical and the red diagonal
  \path[name path=Mvert] (TM) -- (BM);
  \path[name path=Diag]  (BL) -- (TR);
  % intersection point I of the diagonal and the middle vertical
  \path[name intersections={of=Diag and Mvert, by=I}];

  % endpoints of the special (dashed / faded) zone along the diagonal
  \coordinate (AL) at ($(I)!\overlapL!(BL)$);  % toward BL
  \coordinate (AR) at ($(I)!\overlapR!(TR)$);  % toward TR

  % project those points onto the middle vertical to split it in three
  \coordinate (J1) at (TM |- AL);
  \coordinate (J2) at (TM |- AR);

  % ---------------- draw blue verticals ----------------
  \draw[softblue!80!black, line width=\lw] (TL) -- (BL);              % left vertical
  \draw[softblue!80!black, line width=\lw] (TR) -- (BR);              % right vertical

  % middle vertical: strong blue – light blue – strong blue (faded in the overlap)
  \draw[softblue!80!black,           line width=\lw] (TM) -- (J1);
  \draw[softblue!80!black, line width=\lw] (J1) -- (J2);    % lighter tint where the red is dashed
  \draw[softblue!80!black,           line width=\lw] (J2) -- (BM);

  % ---------------- red horizontals ----------------
  \draw[lavender!80!black, line width=\lw] (TL) -- (TM);
  \draw[lavender!80!black, line width=\lw] (BM) -- (BR);

  % ---------------- red diagonal: solid – dashed – solid ----------------
  % ---------------- red diagonal: parabolic arc ----------------
\draw[lavender!80!black, line width=\lw]
  (BL) .. controls ($(BL)!0.5!(TR)+(0,-1.5)$) .. (TR);

  % ---------------- vertices ----------------
  \foreach \P in {TL,TM,TR,BL,BM,BR}
    \fill (\P) circle (\dotrad);

  % ---------------- labels ----------------
  \node[above=2pt]  at (TL) {$\scriptstyle e=(1,2,3)$};
  \node[above=2pt]  at (TM) {$\scriptstyle (2,1,3)$};
  \node[above=2pt]  at (TR) {$\scriptstyle (3,1,2)$};
  \node[below=6pt]  at (BL) {$\scriptstyle (1,3,2)$};
  \node[below=6pt]  at (BM) {$\scriptstyle (2,3,1)$};
  \node[below=6pt]  at (BR) {$\scriptstyle (3,2,1)$};
\end{tikzpicture}
 \caption{Let $G=S_3$,  where $(1,2,3)$ is the identity.  Take the symmetric generating set $S=\{(2,1,3),(1,3,2)\}$ (the transpositions $(12)$ and $(23)$ in cycle notation), so $S^{-1}=S$.  The Cayley graph $\mathrm{Cay}(G,S)$ has vertex set $G$ and an (undirected) edge between $g_1,g_2$ iff $g_2=sg_1$ for some $s\in S$.  We color the edge blue if $g_2=(1,3,2)g_1$ and purple if $g_2=(2,1,3)\,g_1$. The resulting graph is a $2$-regular connected graph on $6$ vertices, hence a single $6$-cycle $C_6$, and its diameter is $3$.} 
 \end{figure}

Fix a cubical partition of $\Omega$ with mesh $\varepsilon>0$, and let
$n=|\Omega|\,\varepsilon^{-d}$ be the number of cubes. The proof of
\Cref{lem: univ.approx.flow} shows that every measure-preserving map can be
approximated in $L^p$ by a map that permutes these cubes. Thus, at scale
$\varepsilon$, the incompressible part reduces to a finite combinatorial problem
on the symmetric group $\mathscr{G}=S_n$.

Let $S\subset S_n$ be any generating set of cube permutations such that each
$s\in S$ is induced by a measure-preserving flow of \eqref{eq: neural.ode}
driven by a piecewise-constant control. After adjoining inverses if necessary,
we may assume that $S=S^{-1}$. For a target permutation $\sigma\in S_n$, the
minimal number of generators needed to realize $\sigma$ is exactly its word
length $\mathrm{dist}_S(e,\sigma)$ in the Cayley graph $\mathrm{Cay}(S_n,S)$.
Consequently, the worst-case number of blocks is governed by
$\mathrm{diam}(\mathrm{Cay}(S_n,S))$. If each generator in $S$ can be
implemented with at most $C$ switches, then
\[
\#\mathrm{switches}\le C\,\mathrm{diam}(\mathrm{Cay}(S_n,S)).
\]
In this sense, the switch-complexity question becomes a diameter problem for
Cayley graphs of $S_n$.

This dependence on the generating set is substantial. For instance, for the
adjacent transpositions
\(S=\{(12),(23),\dots,(n-1\,n)\},
\)
one has
\[
\mathrm{diam}(\mathrm{Cay}(S_n,S))=\binom{n}{2}.
\]
More generally, the diameter of $\mathrm{Cay}(S_n,S)$ is highly sensitive to
the choice of $S$; see, for instance,
\cite{HelfgottSeress2014,Kraft2015}.

\begin{remark}[Width] \label{rem: width}
We focus on width~$1$ perceptrons. 
All constructions extend verbatim to  width~$m$ by replacing \eqref{eq: two.layer.mlp} with
$v(x,\theta)=\sum_{j=1}^m w_j(a_j\cdot x+b_j)_+$, letting $\theta(t)\in\mathbb{R}^{m(2d+1)}$ be piecewise constant, and allowing several neurons to be active in parallel on each time-interval.
While additional width can only increase expressivity and may reduce the number of switches in specific architectures (by parallelizing elementary building blocks), we do not currently know how to leverage width to improve the worst--case switch bounds in a systematic way, so we keep the exposition at width~$1$.
\end{remark}

\begin{remark}[Geodesics]
On the infinite-dimensional group $\mathrm{SDiff}([0,1]^d)$ of measure-preserving diffeomorphisms equipped with the right–invariant $L^2$-distance in the sense of Arnold \cite{arnold1966aif,ebinmarsden1970}, one has 
\[
\|f-g\|_{L^2([0,1]^d)}\le \mathrm{dist}_{\mathrm{SDiff}}(f,g)\le C\|f-g\|_{L^2([0,1]^d)}^{\alpha},
\]
for some $\alpha>0$ and $C\ge1$. The upper bound is due to \cite{shnirelman94Generalized}. On cubes in dimension $d\ge 3$, this was sharpened to $\alpha=1$ in \cite{SchifferZizza2024}.

In this paper we do not seek geodesics (see \Cref{fig:geodesic}). Our aim is a constructive realization of measure‑preserving maps by \eqref{eq: neural.ode}, and the relevant computational cost is the number of switches of the piecewise‑constant control, which is controlled by the diameter of the associated Cayley graph. In contrast to \cite{Zizza2024,SchifferZizza2024}, where Cayley bounds are leveraged to improve the exponent $\alpha$ in Shnirelman’s inequality, here the kinetic energy $\mathrm{dist}_{\mathrm{SDiff}}$ is immaterial—the construction may be far from geodesic.   
\end{remark}

\begin{figure}[h!]
 \centering
\begin{tikzpicture}[line cap=round, line join=round, scale=0.8]
  %---------------- parameters you can tweak ----------------
  \def\X{6.0}          % horizontal distance from mu to T#mu
  \def\eps{1}        % vertical gap: |(T_\eps)#mu - T#mu|
  \def\H{4}          % "bulge" height parameter for the parabola (increase for a higher arc)
  \def\lw{1.5pt}       % line width for strokes
  \def\dotrad{3pt}   % radius of black dots

  %---------------- coordinates ----------------
  \coordinate (mu)   at (0,0);
  \coordinate (Tmu)  at (\X,0);
  \coordinate (Temu) at (\X,\eps);

  %---------------- red parabola: from mu to Temu ----------------
  % parametric: x = X t, y = eps t + 4 H t (1 - t),  t in [0,1]
  \draw[softblue!80!black, line width=\lw, samples=120, smooth, variable=\t, domain=0:1]
    plot ({\X*\t}, {\eps*\t + 4*\H*\t*(1-\t)});

  %---------------- green straight segment (mu to T#mu) ----------------
  \draw[lavender!80!black, line width=\lw] (mu) -- (Tmu);

  %---------------- dotted epsilon segment on the right ----------------
  \draw[gray, dashed, line width=1pt] (Tmu) -- (Temu);
  \node[left=4pt] at ($(Tmu)!0.5!(Temu)$) {$\varepsilon$};

  %---------------- black vertices ----------------
  \fill (mu)   circle (\dotrad);
  \fill (Tmu)  circle (\dotrad);
  \fill (Temu) circle (\dotrad);

  %---------------- labels ----------------
  \node[left=4pt] at (mu)   {$\mu$};
  \node[right=6pt]      at (Tmu)  {$\phi_{\#}\mu$};
  \node[right=6pt]      at (Temu) {$\phi_{\varepsilon\#}\mu$};

  % optional: bounding box padding
  \path[use as bounding box] (-0.6,-0.8) rectangle (\X+1.3,\eps+1.3);
\end{tikzpicture}
\caption{ \Cref{thm: gen.thm} allows us to approximate, in particular, the optimal transport map $\phi$ between $\mu$ and a target absolutely continuous measure $\nu$ by an approximate map $\phi_\varepsilon$. However, even if the densities $\phi_{\#}\mu$ and $\phi_{\varepsilon\#}\mu$ are close in $\mathsf{TV}$ and the maps are close, the trajectory $t\mapsto\mu^\varepsilon(t)$ given by \eqref{eq: cont.eq} is not close to $((1-t)\mathsf{id}+t\phi)_\#\mu$ for all times $t$.}%In $d=1$ we can approximate the increasing rearrangement $\phi_{\text{IR}}$ (the optimal transport map) by $\phi_{\text{IR}}^\varepsilon$ and  $\phi_{\text{IR}\#}\mu$ by $\phi_{\text{IR}\#}^\varepsilon\mu$. However, the trajectory $t\mapsto\mu^\varepsilon(t)$ given by \eqref{eq: cont.eq} is not close to $((1-t)\mathsf{id}+t\phi_{\text{IR}})_\#\mu$ for all times $t$.}
\label{fig:geodesic}
\end{figure}
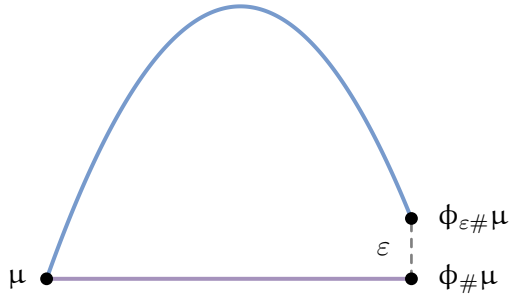

\subsubsection{Further related work}

Using control theory techniques to understand various aspects of neural networks has been a fruitful line of research. See, for instance, \cite{weinan2017proposal, esteve2020large, geshkovski2021control, geshkovski2022turnpike, esteve2023sparsity, cipriani2024minimax, daudin2025genericity, puttschneider2026dissipativity, hintermuller2026layerwise, grong2026controllability} for optimal control, and \cite{li2022deep, elamvazhuthi2022neural, ruiz2023neural, ishikawa2023universal, li2024deep, scagliotti2025minimax, scagliotti2025normalizing, de2025optimal, li2026universal} for constructive approximate control approaches. %{\color{teal}   \cite{rouhvarzi2025incremental} study the autonomous incremental setting, proving that finite compositions of autonomous flows are universal for orientation-preserving homeomorphisms of $[0,1]^d$; their extensions to arbitrary continuous maps and probability measures rely on a linear lift to one additional dimension.}
Results of this nature have also been adapted to more advanced neural network architectures such as Transformers \cite{geshkovski2025mathematical, geshkovski2024measure, adu2024approximate, agrachev2025generic}.

The recent work \cite{kazandjian2025orbits} also focuses on the approximate controllability question for the continuity equation, but only considers the approximation of a target density and restricts attention to Hamiltonian flows. In particular, the proof of approximating the measure-preserving components is very similar (and is alike the strategy of \cite{brenier2003approximation}), although they must use Hamiltonian flows instead of shear flows, as we do.

Finally, a related work on the Maurey-in-time approach used in the proof of \Cref{thm:sobolev_to_switches_TV} is \cite{ma2022barron}. However, the approach is different, as the authors work with discrete-time updates and define the limiting object as an accumulated integral function class. Moreover, they do not use log-Jacobian tracking updates to obtain a joint approximation of a function and the density.

\section*{Acknowledgments}
We thank Yann Brenier, Aldo Pratelli, Carlos Mora and Kimi Sun for motivating discussions.
B.G.'s research was kindly supported by a Sorbonne Emergences grant and a gift by Google.
D.R.B. research was supported by “France 2030” support managed by the Agence Nationale de la Recherche, under the reference ANR-23-PEIA-0004.

\section{Proof of \Cref{thm: gen.thm}} \label{sec: proof.th1}

In \Cref{subsec: lagrange}, we approximate the target diffeomorphism $\phi$ by a piecewise--affine Lagrange interpolant $\phi_\varepsilon$ (and simultaneously approximate $\phi^{-1}$) using \Cref{thm: lagrange.approx} from \cite{iwaniec2017triangulation}. Then, we then exploit the special structure of Lagrange interpolants: \Cref{prop: lagrange.decomposition} yields a polar-like factorization $\phi_\varepsilon=m_{2\varepsilon}\circ g_\varepsilon\circ m_{1\varepsilon}$ into two measure--preserving pieces $m_{1\varepsilon},m_{2\varepsilon}$ and a simple compressible map $g_\varepsilon=\nabla\varphi$ acting only on one coordinate and given by a monotone piecewise--affine profile. In \Cref{subsec: approx.flows}, we realize these factors by neural ODE flows with piecewise--constant controls: the compressible factor is implemented \emph{exactly} by \Cref{lem: exact.linear.nonuniform}, while the incompressible factors are approximated by composing explicit divergence--free ``swap'' flows (constructed in \Cref{lem: mp.neural.ode.swap}) inside the permutation scheme of \Cref{lem: univ.approx.flow}. 
Finally, in \Cref{subsec: stab.estimates}, we propagate the resulting $L^p$ approximation of maps to $\mathsf{TV}$--control of pushforwards using the stability tools in \Cref{lem: contraction} and \Cref{thm: stability.measure.preserving}; assembling these ingredients yields first \eqref{eq: Lp.small} and then \eqref{eq: TV.small}, completing the proof.

\subsection{Approximation cornerstones}

\subsubsection{Lagrange interpolation}\label{subsec: lagrange}

We work with Lagrange interpolations throughout, which calls for a remainder of several concepts. 

We call a family of subsets $\mathscr{F}$ of $\R^d$ a \emph{partition} of a domain $\Omega\subset\R^d$ if $\mathscr{F}$ consists of pairwise non-overlapping (namely, have disjoint interiors) closed subsets of $\Omega$ such that $\bigcup_{A\in\mathscr{F}}A=\Omega$. A partition $\mathscr{F}$ is \emph{simplical}---henceforth referred to as a \emph{triangulation}---provided $\mathscr{F}=\{\triangle_j\}_{j\geq1}$, where $\triangle_j$ are closed $d$-dimensional simplexes such that
1). every compact subset of $\Omega$ intersects only a finite number of simplexes in $\mathscr{F}$; 2). the intersection of two simplexes in $\mathscr{F}$ is either empty or a $k$-dimensional subsimplex of both simplexes, $k=0,\ldots,d$.\footnote{The common convention is that vertices are $0$-dimensional subsimplexes and the simplex itself is its $d$-dimensional subsimplex.}
A continuous map $f_{\mathscr{L}}:\Omega\to\R^d$ is called \emph{piecewise affine} if there is a triangulation $\mathscr{F}=\{\triangle_j\}_{j\geq1}$ of $\Omega$ such that $f_{\mathscr{L}}$ is affine on each simplex $\triangle_j$. Such a map $f_{\mathscr{L}}$ is called a \emph{Lagrange interpolation} of a continuous map $f:\Omega\to\R^d$ if it agrees with $f$ at the vertices of every simplex of the triangulation. 

\subsubsection*{Piecewise affine approximation}

The following result is shown in \cite{iwaniec2017triangulation}.

\begin{theorem}[\cite{iwaniec2017triangulation}] \label{thm: lagrange.approx}
    Let $\Omega_1, \Omega_2\subset\R^d$ be two domains, suppose $\phi:\Omega_1\to\Omega_2$ is a $C^1$--smooth diffeomorphism, and let $\psi$ denote its inverse. Let  $\varepsilon>0$. There exists a piecewise affine homeomorphism $\phi_\varepsilon:\Omega_1\to\Omega_2$ (actually a Lagrange interpolation of $\phi$) which satisfies 
    \begin{equation*}
        |\phi(x)-\phi_\varepsilon(x)|+|\nabla\phi(x)-\nabla\phi_\varepsilon(x)|\leq\varepsilon \hspace{1cm} \text{ for a.e. }x\in\Omega_1.
    \end{equation*}
    Its inverse $\psi_\varepsilon:\Omega_2\to\Omega_1$ satisfies
    \begin{equation*}
        |\psi(x)-\psi_\varepsilon(x)|+|\nabla\psi(x)-\nabla\psi_\varepsilon(x)|\leq\varepsilon \hspace{1cm} \text{ for a.e. }x\in\Omega_2.
    \end{equation*}
\end{theorem}

To prove \Cref{thm: lagrange.approx}
\cite{iwaniec2017triangulation} construct a specific selfsimilar isotropic triangulation of $\Omega_1$. First, the authors decompose $\Omega_1$ into a dyadic grid made of dyadic cubes, meaning hypercubes obtained by recursive bisection. This dyadic structure is locally isotropic: all cubes have comparable side lengths relative to their position, avoiding highly elongated or degenerate shapes near the boundary. 
Next, the dyadic cubes are further triangulated into simplexes in a selfsimilar way. Each dyadic cube is partitioned into a controlled number of simplexes by subdividing it into model simplexes that are rescaled versions of a finite set of prototypes. This ensures that the resulting simplexes are uniformly regular: all simplexes have bounded aspect ratios, avoiding flattening or degeneration. Importantly, the triangulation respects the dyadic structure: faces of adjacent cubes are triangulated compatibly, and the construction is stable under dyadic refinement. This carefully designed triangulation enables a Lagrange interpolation that approximates $\phi$ both uniformly and in derivative, while preserving injectivity.

\begin{remark}
Relaxing the $C^1$ assumption to Sobolev homeomorphisms or diffeomorphisms is already a delicate problem even in low dimension \cite{daneri2014smooth,mora2014approximation,iwaniec2017limits}. 
In some settings such approximations can fail altogether: \cite{donaldson1989quasiconformal} 
exhibit an obstruction in dimension $4$ to approximation in $C^0$ norm of homeomorphisms along with their inverses by piece-wise constant approximations.
\end{remark}

\subsubsection*{Polar-like factorization of Lagrange interpolants}

Let $\Omega\subseteq\R^d$. 
We say that $m:\Omega\to\Omega$ is \emph{measure-preserving} (with respect to the Lebesgue measure) if 
\begin{equation*}
    \int_A \diff x = \int_{m(A)}\diff x \hspace{1cm} \text{ for all } A\subset\Omega.
\end{equation*}
Continuously differentiable diffeomorphisms $m$ are measure-preserving if and only if $|\det \nabla m|\equiv 1$.
It is well-known that general vector-valued functions admit a polar factorization as the composition of the gradient of a convex function with a measure-preserving map \cite{brenier1991polar}. In the special case of Lagrange interpolants, slightly more can be said---they can be decomposed in three "actions", which impel the geometrical nature of the flows we construct subsequently.

\begin{figure}[h!]
    \centering
    \[
    \includegraphics[scale=0.7]{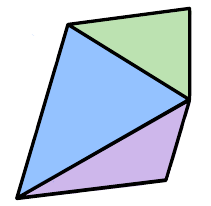}
    \quad \raisebox{1cm}{\(\xmapsto{\hspace{0.5cm}\phi_{\mathscr{L}}\hspace{0.5cm}}\)} \quad
    \includegraphics[scale=0.7]{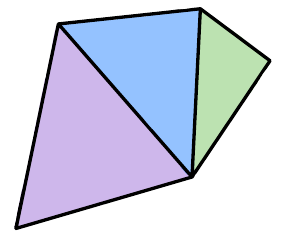}
    \]
    \caption{A triangulation and its image by a piecewise affine map $\phi_{\mathscr{L}}$.}
    \label{fig:placeholder}
\end{figure}

\begin{proposition} \label{prop: lagrange.decomposition}
Let $\Omega\subset\R^d$ be a bounded orthotope.
Let $\phi_{\mathscr{L}}:\Omega\to\phi_{\mathscr{L}}(\Omega)\subset\R^d$ denote the Lagrange interpolation of a $C^1$-diffeomorphism $\phi:\Omega\to\phi(\Omega)\subset\R^d$ as per \Cref{thm: lagrange.approx}. 
Then 
\begin{equation*}
    \phi_{\mathscr{L}} = m_{2}\circ\nabla\varphi\circ m_1 \hspace{1cm} \text{ on } \Omega,
\end{equation*}
where 
\begin{enumerate}
    \item $m_1,m_2$ are measure-preserving measurable bijections of $\R^d$ with measurable inverses, are of bounded variation on compact sets (indeed piecewise affine on a polyhedral partition), and satisfy $|\det\nabla m_j|=1$ a.e., as well as $|\det\nabla m_j^{-1}|=1$ a.e.;
    \item $\nabla\varphi$ has the special form
    \[
    \nabla\varphi(x_1,\dots,x_d)=\psi'(x_1)e_1+\sum_{k=2}^d x_k e_k,
    \]
    where $\psi':\R\to\R$ is continuous, strictly increasing, and piecewise affine;
    in particular $\nabla\varphi$ is a bijection of $\R^d$.
\end{enumerate}
\end{proposition}

\begin{figure}[h!]
    \centering
    \[
    \includegraphics[scale=0.5]{Figures/triangle-1.pdf}
    \quad \raisebox{1cm}{\(\xmapsto{\hspace{0.2cm}m_1\hspace{0.2cm}}\)} \quad
    \includegraphics[scale=0.5]{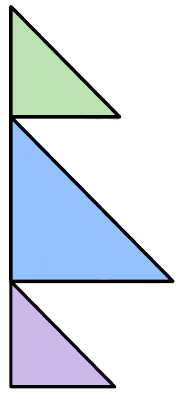}
    \quad \raisebox{1cm}{\(\xmapsto{\hspace{0.2cm}\nabla\varphi\hspace{0.2cm}}\)} \quad
    \includegraphics[scale=0.5]{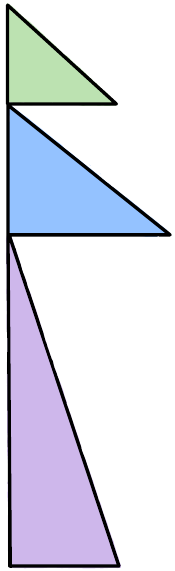}
    \quad \raisebox{1cm}{\(\xmapsto{\hspace{0.2cm}m_2\hspace{0.2cm}}\)} \quad
    \includegraphics[scale=0.5]{Figures/triangle-2.pdf}
    \]
    \caption{The decomposition of $\phi_{\mathscr{L}}$ per \Cref{prop: lagrange.decomposition}.}
    \label{fig:placeholder}
\end{figure}

\begin{proof}
Since $\phi_{\mathscr L}$ is piecewise affine on $\Omega$ and $\Omega$ is bounded, we may write
\[
\Omega=\bigcup_{j=1}^n \triangle_j,
\]
where $\{\triangle_j\}_{j=1}^n$ is a finite triangulation of $\Omega$ into non-overlapping $d$-simplices
(we ignore boundaries, which are null sets).
On each simplex $\triangle_j$, the map $\phi_{\mathscr L}$ is affine:
\[
\phi_{\mathscr L}(x)=A_jx+b_j\qquad x\in \triangle_j,
\]
and since $\phi_{\mathscr L}$ is a homeomorphism onto its image, each $A_j$ is invertible and
$|\phi_{\mathscr L}(\triangle_j)|>0$.

\medskip
\noindent\textbf{Step 1: A tower.}
Let $V_j\coloneqq |\triangle_j|>0$, and set $s_j\coloneqq V_j^{1/d}$.
Define partial sums $H_j\coloneqq \sum_{i<j}s_i$ and $H\coloneqq \sum_{j=1}^n s_j$.
Fix a reference simplex $\triangle_{\mathrm{ref}}\subset\R^d$ of volume $1$ with the property that
\[
\mathrm{proj}_{x_1}(\triangle_{\mathrm{ref}})=[0,1].
\]
Choose $L>0$ so large that the interval
\[
S\coloneqq [L,L+H+1]\times[-1,1]^{d-1}
\]
is disjoint from $\Omega$ and also disjoint from $\phi_{\mathscr L}(\Omega)$.
For each $j$ define the auxiliary simplex
\[
\triangle'_j \coloneqq  s_j\,\triangle_{\mathrm{ref}} + (L+H_j)e_1.
\]
Then $|\triangle'_j|=s_j^d|\triangle_{\mathrm{ref}}|=V_j=|\triangle_j|$, and
the $x_1$-projections are pairwise disjoint:
\[
I_j\coloneqq \mathrm{proj}_{x_1}(\triangle'_j)=[L+H_j,\ L+H_j+s_j),
\qquad I_j\cap I_i=\varnothing\ \text{ for }i\neq j.
\]
Set $T\coloneqq \bigcup_{j=1}^n \triangle'_j\subset S$.

\medskip
\noindent\textbf{Step 2: Constructing $m_1$.}
For each $j$ choose an affine bijection
\[
B_j:\triangle_j\to \triangle'_j
\]
sending vertices to vertices, and choose the vertex ordering so that $\det\nabla B_j=+1$
(this is always possible since $|\triangle_j|=|\triangle'_j|$).
Define $m_1:\R^d\to\R^d$ by the ``swap'' rule
\[
m_1(x)\coloneqq 
\begin{cases}
B_j(x), & x\in \triangle_j,\\
B_j^{-1}(x), & x\in \triangle'_j,\\
x, & x\notin \Omega\cup T.
\end{cases}
\]
Because $\Omega$ and $T$ are disjoint and each $B_j$ is bijective,
$m_1$ is a measurable bijection of $\R^d$ with measurable inverse (given by the same formula).
Moreover $m_1$ is affine on each polyhedral piece, hence $m_1\in  \mathsf{BV}_{\mathrm{loc}}(\R^d;\R^d)$,
and $|\det\nabla m_1|=1$ a.e. (indeed it equals $1$ on each affine piece).
Therefore $m_1$ preserves Lebesgue measure.

\medskip
\noindent\textbf{Step 3: The compressible factor $g=\nabla\varphi$.}
Define the target Jacobian factors
\[
\lambda_j\coloneqq \frac{|\phi_{\mathscr L}(\triangle_j)|}{|\triangle_j|}=\frac{|\phi_{\mathscr L}(\triangle_j)|}{V_j}>0.
\]
We now build $\psi:\R\to\R$ convex with $\psi'$ continuous, strictly increasing and piecewise affine,
such that
\[
\psi''(x_1)=\lambda_j\quad\text{for }x_1\in I_j,
\qquad
\psi''(x_1)=1\quad\text{for }x_1\notin \bigcup_{j=1}^n I_j,
\]
where $\psi''$ is understood in the a.e.\ sense (piecewise constant).
Fix the normalization $\psi'(x_1)=x_1$ for $x_1\le L-1$; integrating $\psi''$ produces a continuous,
strictly increasing, piecewise affine function $\psi'$ (and a convex $\psi$).

Define
\[
\varphi(x)\coloneqq \psi(x_1)+\frac12\sum_{k=2}^d x_k^2,
\qquad
g(x)\coloneqq \nabla\varphi(x)=(\psi'(x_1),x_2,\dots,x_d).
\]
Then $\det\nabla g(x)=\psi''(x_1)$ a.e. In particular, on each $\triangle'_j$ we have $x_1\in I_j$,
hence $\det\nabla g\equiv \lambda_j$ a.e.\ on $\triangle'_j$, and therefore
\[
|g(\triangle'_j)|=\int_{\triangle'_j}\det\nabla g\diff x=\lambda_j|\triangle'_j|
=\lambda_jV_j=|\phi_{\mathscr L}(\triangle_j)|.
\]
Set $\hat\triangle_j\coloneqq g(\triangle'_j)$ and $\hat T\coloneqq \bigcup_j\hat\triangle_j$.
By taking $L$ large enough initially, we may assume $\hat T$ is disjoint from $\phi_{\mathscr L}(\Omega)$.

\medskip
\noindent\textbf{Step 4: Constructing $m_2$.}
For each $j$, let $v_0^{(j)},\dots,v_d^{(j)}$ be the vertices of $\triangle_j$ and define
\[
y_i^{(j)}\coloneqq \phi_{\mathscr L}(v_i^{(j)})\qquad(i=0,\dots,d).
\]
Let $\hat v_i^{(j)}\coloneqq g(B_j(v_i^{(j)}))$ be the corresponding vertices of $\hat\triangle_j$.
Since $|\hat\triangle_j|=|\phi_{\mathscr L}(\triangle_j)|$, there exists a unique affine map
\[
C_j:\hat\triangle_j\to \phi_{\mathscr L}(\triangle_j)
\]
such that $C_j(\hat v_i^{(j)})=y_i^{(j)}$ for all $i$, and again we choose the vertex ordering so
that $\det\nabla C_j=+1$.

Define $m_2:\R^d\to\R^d$ by
\[
m_2(x)\coloneqq 
\begin{cases}
C_j(x), & x\in \hat\triangle_j,\\
C_j^{-1}(x), & x\in \phi_{\mathscr L}(\triangle_j),\\
x, & x\notin \hat T\cup \phi_{\mathscr L}(\Omega).
\end{cases}
\]
As before, this is a measurable bijection with measurable inverse, it lies in $ \mathsf{BV}_{\mathrm{loc}}$,
satisfies $|\det\nabla m_2|=1$ a.e.\ and preserves Lebesgue measure.

\medskip
\noindent\textbf{Step 5: The factorization on $\Omega$.}
Fix $j$ and $x\in\triangle_j$. Then $m_1(x)=B_j(x)\in\triangle'_j$,
hence $g(m_1(x))\in\hat\triangle_j$, and finally $m_2(g(m_1(x)))\in\phi_{\mathscr L}(\triangle_j)$.
Moreover, on $\triangle_j$ the map $m_2\circ g\circ m_1$ is affine:
$m_1$ is affine on $\triangle_j$, $g$ is affine on $\triangle'_j$ because $\psi'$ is affine on $I_j$,
and $m_2$ is affine on $\hat\triangle_j$.
By construction it maps each vertex $v_i^{(j)}$ to $y_i^{(j)}=\phi_{\mathscr L}(v_i^{(j)})$.
Since two affine maps that agree on the $d+1$ vertices of a simplex agree everywhere on that simplex,
we conclude that
\[
(m_2\circ g\circ m_1)(x)=\phi_{\mathscr L}(x)\qquad\text{for all }x\in\triangle_j.
\]
As $j$ was arbitrary, this holds on all of $\Omega$, completing the proof.
\end{proof}

\subsubsection{Approximation through flows}\label{subsec: approx.flows}

\subsubsection*{Elementary flows}

The following lemmas will be of great use in several subsequent arguments.

\begin{lemma}\label{lem: 1d} Fix $T>0$. 
    \begin{enumerate}
        \item Let $d=1$. The solution to \eqref{eq: neural.ode} for $a\equiv1$ is
        \begin{equation*}
            \upphi_{(w,b)}^t(x)= \left(e^{wt}(x-b)+b\right)\mathbf{1}_{x\geq b} + x\mathbf{1}_{x\leq b}
        \end{equation*}
        and for $a\equiv-1$,
        \begin{equation*}
            \upphi_{(w,b)}^t(x)=\left(e^{-wt}(x-b)+b\right)\mathbf{1}_{x\leq b} + x\mathbf{1}_{x\geq b}.
        \end{equation*}
        \item Let $h,\uptau\geq0$ and $c\in \mathbb{R}$. There exists a flow map $\upvarphi_{\theta}^t$ of \eqref{eq: neural.ode} corresponding to piecewise constant $\theta=(w,a,b)$ with $2$ switches such that
        \begin{equation*}
            \upvarphi_{\theta}^T(x)=(x+\uptau)\mathbf{1}_{x\geq c+h}+x\mathbf{1}_{x\leq c},
        \end{equation*}
       with $a\equiv 1$ and
        \begin{equation*}
        \|w\|_{L^\infty(0,T)}\leq \frac{1}{T}\log\left(1+\frac{\uptau}{h}\right),\qquad \|b\|_{L^\infty(0,T)}\leq c+\uptau+h.
        \end{equation*}
    \end{enumerate}
\end{lemma}

\begin{proof}
    The first point is straightforward.
    
    As for the second, let $\upvarphi_{\theta}^t=\upphi_{(-w,c+\upeta)}^t\circ\upphi_{(w,c)}^t$
    with $\upeta>0$ to be determined later on. Elementary computations yield
    \begin{equation*}
        \left(\upphi_{(-w,c+\upeta)}^t\circ\upphi_{(w,c)}^t\right)(x)= x+\left(1-e^{-wt}\right)\upeta\hspace{1cm}\text{ if }x\geq c+e^{-wt}\upeta.
    \end{equation*}
   We are thus looking for a solution to the system $h=e^{-wt}\upeta$ and $\uptau=(1-e^{-wt})\upeta$
    that is to say, $\uptau+h=\upeta$ and $\log(\frac{h}{h+\uptau})=-wt$.
    Since $c+h>c$, for $x\leq b$ one has $(\upphi_{(-w,c+\upeta)}^t\circ\upphi_{(w,c)}^t)(x)=x$. The estimates easily follow.
\end{proof}

\begin{lemma}\label{lem: several.d}
Let $d\ge 2$. Fix $h>0$, $\uptau,b\in\mathbb{R}$, $a\in\{-1,1\}$, and coordinate vectors $e_k,e_l$ with $l\neq k$.
Set $n\coloneqq a\,e_k$ and $u\coloneqq \operatorname{sgn}(\uptau)\,e_l$. Then there exists a piecewise-constant control $\theta:[0,T]\to\mathbb{R}^{2d+1}$ with one switch such that the time-$T$ flow map $\upvarphi^T_\theta$ of \eqref{eq: neural.ode} satisfies
\[
  \upvarphi^T_\theta(x)=
  \begin{cases}
    x+\uptau\,e_l & \text{if } x\cdot n + b - h \ge 0,\\[2mm]
    x             & \text{if } x\cdot n + b \le 0.
  \end{cases}
\]
\end{lemma}

\begin{proof}
For $t\ge 0$ and $\beta\in\mathbb{R}$, define the shear
\[
  \upphi^{t}_{(n,u,\beta)}(x)\coloneqq x + (x\cdot n+\beta)_+ tu.
\]
Take $t\coloneqq |\uptau|/h$ and set $T\coloneqq 2t$. Consider $\upvarphi^T_\theta \;\coloneqq\; \upphi^{t}_{(n,u,b)} \circ \upphi^{t}_{(n,-u,b-h)}$.
This corresponds to the piecewise-constant
\[
  \theta(s)=
  \begin{cases}
    (w,a,b)=( -u,\, n,\, b-h), & s\in[0,t),\\
    (w,a,b)=( \phantom{-}u,\, n,\, b),   & s\in[t,2t)= [t,T).
  \end{cases}
\]
Since $u\perp n$ (because $e_l\cdot e_k=0$ for $l\neq k$), along each stage the quantity $x(s)\cdot n$ is constant:
\[
  \frac{\mathrm{d}}{\mathrm{d}s}(x(s)\cdot n+\beta)
  = \dot x(s)\cdot n
  = (x(s)\cdot n+\beta)_+\,u\cdot n
  = 0,
\]
so the activation depends only on the initial $x$ for that stage. Therefore, for all $x$,
\[
  \upvarphi^T_\theta(x)
  = x + \left[(x\cdot n+b)_+ - (x\cdot n+b-h)_+\right] tu.
\]
Analyzing the three regions:
\begin{itemize}
\item If $x\cdot n+b\ge h$, both terms are positive and their difference is $h$, hence
$\upvarphi^T_\theta(x)=x+h\,t\,u=x+\uptau\,e_l$.
\item If $x\cdot n+b\le 0$, both terms vanish and $\upvarphi^T_\theta(x)=x$.
\item If $0<x\cdot n+b<h$, only the first term is positive, yielding
$\upvarphi^T_\theta(x)=x+\frac{|\uptau|}{h}(x\cdot n+b)\,u$, which lies in the (unspecified) middle strip.
\end{itemize}
Finally, each stage is the time-$t$ flow of the vector field $x\mapsto \pm u\,(x\cdot n+\beta)_+$, whose divergence is
\[
  \nabla_x\cdot(u\,(x\cdot n+\beta)_+)
  = u\cdot n\,\mathbf{1}_{\{x\cdot n+\beta>0\}}=0,
\]
so Lebesgue measure is preserved at a.e.\ time; by Liouville’s formula, the Jacobian determinant of each stage equals $1$, and hence the composition is measure preserving.
\end{proof}

\subsubsection*{Incompressible part}

\begin{lemma} \label{lem: univ.approx.flow}

Suppose $\Omega\subset\R^d$ is a bounded domain.
Let $m:\mathbb{R}^d\to\mathbb{R}^d$ be a bijective measure-preserving map. 
For any $\varepsilon>0$ there exists a flow $m_\theta^1:\R^d\to\R^d$ of \eqref{eq: neural.ode} associated to piecewise-constant $\theta$ with finitely many switches, which is measure-preserving, and such that
    \begin{equation*}
        \|m-m_\theta^1\|_{L^p(\Omega)}\leq \varepsilon.
    \end{equation*}
\end{lemma}

\begin{proof}
The argument is a standard scheme going back to \cite{shnirel1987geometry, shnirelman94Generalized} and \cite{brenier2003approximation}. For clarity we first treat the case $\Omega=[-L,L]^d$ as in
\Cref{lem: mp.neural.ode.swap}; a general bounded domain can be covered by such a
box and the divergence-free vector fields we build can be taken to vanish outside
$\Omega$. We also implicitly restrict attention to the natural setting in which
$m$ maps $\Omega$ onto itself a.e.; a routine cut-and-paste on a set of arbitrarily
small measure reduces the general case to this one.

\smallskip
\noindent\textbf{Step 1. Cores.}
Fix $\varepsilon>0$ and $1\le p<\infty$. By Lusin’s theorem there exists a
compact set $K\subset\Omega$ such that $|\Omega\setminus K|\le\eta$ and $m|_{K}$ is uniformly continuous, where $\eta>0$ will be chosen later. 
Let $\omega$ be a modulus of continuity for
$m|_K$.
Choose a mesh size $h>0$ so small that
\[
\omega(h\sqrt d)\le\frac{\varepsilon}{8}
\quad\text{and}\quad
|\{x\in\Omega\colon \operatorname{dist}(x,\partial\Omega)<h\}|\le\eta.
\]
Let $h>\delta>0$ be so small that
\[
4\omega(h\sqrt d)\le\delta
\quad\text{and}\quad
\left|\bigcup_{i\in h\mathbb Z^d\cap\Omega}
\left((i+[0,h]^d)\setminus\square_{ih\delta}\right)\right|\le \eta,
\]
where $\square_{ih\delta}\coloneqq [i_1,i_1+h-\delta]\times\cdots\times[i_d,i_d+h-\delta]$.
Set
\[
U\coloneqq\bigcup_{i\in h\mathbb Z^d\cap\Omega}\square_{ih\delta}.
\]
Then, for $h,\delta$ sufficiently small (depending only on $\eta$ and $d$),
\begin{equation}\label{eq:U-large}
|\Omega\setminus U|\le C_d\left(\frac{\delta}{h}+h\right)|\Omega|\le 2\eta.
\end{equation}

\noindent\textbf{Step 2: On a large set, $m$ sends each atom into the core of a single atom.}
For each $j$, define the shrunken core
\[
\square_{jh\delta}^{\circ}\coloneqq\left\{y\in \square_{jh\delta}\colon
\operatorname{dist}(y,\partial \square_{j,h,\delta})\ge \frac{\delta}{4}\right\}.
\]
Because $m|_K$ is uniformly continuous and $4\omega(h\sqrt d)\le\delta$,
for every index $i$ with $K_i\coloneqq K\cap \square_{ih\delta}\neq\varnothing$ we
have $\operatorname{diam}(m(K_i))\le \delta/4$. Hence, either $m(K_i)\subset \square_{jh\delta}^{\circ}$ for a
unique $j$ (\emph{good}), or $m(K_i)$ intersects the grid stripes
$\Omega\setminus\bigcup_j \square_{jh\delta}^{\circ}$ (\emph{bad}).

Let
\[
G\coloneqq\bigcup_{\text{good }i}K_i,\qquad B\coloneqq\Omega\setminus G.
\]
Since $m$ is measure-preserving,
\begin{align*}
|B|\le|\Omega\setminus K|+
\left|m^{-1}\left(\Omega\setminus\bigcup_j\square_{jh\delta}^{\circ}\right)\right|=|\Omega\setminus K|+\left|\Omega\setminus\bigcup_j\square_{jh\delta}^{\circ}\right|\le 3\eta.
\end{align*}
For each good $i$ there is a unique $j=j(i)$ with
$m(K_i)\subset\square_{j(i)h\delta}^{\circ}$. Injectivity of $m$ implies that
$i\mapsto j(i)$ is injective on the set of good indices. Extend this to a full permutation $\sigma$ of all indices by pairing the remaining indices arbitrarily. Define the block-permutation map $P_\sigma$ by translating each
$\square_{ih\delta}$ onto $\square_{\sigma(i)h\delta}$ (and setting
$P_\sigma=\mathsf{id}$ on $\Omega\setminus U$). This $P_\sigma$ is
measure-preserving and bijective.

\smallskip
\noindent\textbf{Step 3: $L^p$ control of $m-P_\sigma$.}
Split $\Omega=G\cup B$.
If $x\in G\cap\square_{ih\delta}$ with $i$ good, then $m(x)$ and
$P_\sigma(x)$ both lie in $\square_{\sigma(i)h\delta}$, hence
\[
|m(x)-P_\sigma(x)|\le \operatorname{diam}(\square_{\sigma(i)h\delta})\le C_d\,h.
\]
Therefore
\[
\int_{G} |m-P_\sigma|^p \diff x \le (C_d h)^p |\Omega|.
\]
On the bad set $B$ we use the trivial bound
$|m(x)-P_\sigma(x)|\le 2\operatorname{diam}(\Omega)$ to get
\[
\int_{B} |m-P_\sigma|^p \diff x \le
(2\operatorname{diam}(\Omega))^p |B|
\le (2\operatorname{diam}(\Omega))^p 3\eta.
\]
Hence
\begin{equation}\label{eq:first-approx}
\|m-P_\sigma\|_{L^p(\Omega)}^p
\le (C_d h)^p|\Omega| + (2\operatorname{diam}(\Omega))^p 3\eta.
\end{equation}
We will choose $h$ and $\eta$ so that the right-hand side is $\le
(\varepsilon/2)^p$.

\medskip
\noindent\textbf{Step 4: Realize the permutation by finitely many swaps.}
Decompose $\sigma$ into disjoint cycles\footnote{A cycle is  a permutation of the form $(a_1,a_2,\dots,a_k)$ which maps $a_i\mapsto a_{i+1}$ for $i\in\{1,\dots,k-1\}$, maps $a_k\mapsto a_1$, and fixes every other element. A transposition is a $2$-cycle $(a,b)$. Two cycles are disjoint if they move disjoint sets of elements, in which case they commute. Every permutation decomposes uniquely (up to the order of the factors) into a product of disjoint cycles. Each $k$-cycle is a product of $k-1$ transpositions, e.g.
\(
(a_1,a_2,\dots,a_k)=(a_1,a_k)(a_1,a_{k-1})\cdots(a_1,a_2).
\)
We compose permutations right-to-left.}. Each $k$-cycle is a product of $k-1$
transpositions. By \Cref{lem: mp.neural.ode.swap}, each transposition
$\square_{ih\delta}\leftrightarrow \square_{jh\delta}$ is the time–one flow map
of \eqref{eq: neural.ode} with a divergence-free velocity
$v(\cdot,\theta(t))$, where $\theta(t)$ is piecewise constant over that time
segment, and all other cubes $\square_{kh\delta}$ are fixed as sets. By
concatenating the finitely many transpositions (reparameterizing time to $[0,1]$ if desired) we obtain a piecewise-constant control $\theta(t)$ with finitely
many switches whose time–one flow $\upphi\coloneqq m_\theta^1$ satisfies $\upphi(\square_{ih\delta})=\square_{\sigma(i)h\delta}$ for all $i$ and $\nabla\cdot v(\cdot,\theta(t))=0$ for all $t$.
In particular, $\upphi$ is a bijective, orientation- and measure-preserving map.

If $x\in\square_{ih\delta}$ then $\upphi(x),P_\sigma(x)\in
\square_{\sigma(i)h\delta}$, so
\[
|\upphi(x)-P_\sigma(x)|\le \operatorname{diam}(\square_{\sigma(i)h\delta})
\le C_d\,h .
\]
For $x\in \Omega\setminus U$ we use the trivial bound
$|\upphi(x)-P_\sigma(x)|\le 2\,\operatorname{diam}(\Omega)$ and
\eqref{eq:U-large}. Hence
\begin{align}\label{eq:second-approx}
\|\upphi-P_\sigma\|_{L^p(\Omega)}^p
&\le (C_d h)^p |U| + (2\operatorname{diam}(\Omega))^p |\Omega\setminus U|\nonumber\\
&\le (C_d h)^p |\Omega| + (2\operatorname{diam}(\Omega))^p 2\eta.
\end{align}

\noindent\textbf{Step 5: Conclusion.}
By the triangle inequality, \eqref{eq:first-approx} and \eqref{eq:second-approx},
\[
\|m-\upphi\|_{L^p(\Omega)}
\le \|m-P_\sigma\|_{L^p(\Omega)} + \|P_\sigma-\upphi\|_{L^p(\Omega)}
\le C_d h|\Omega|^{\frac1p} + C\eta^{\frac1p}.
\]
Choose $h$ small so that $C_d h|\Omega|^{1/p}\le\varepsilon/2$, and
$\eta$ so that $C\eta^{1/p}\le \varepsilon/2$. With this choice,
$\upphi=m_\theta^1$ is a measure-preserving time–one flow map of
\eqref{eq: neural.ode} driven by a piecewise-constant control with finitely many
switches, and $\|m-m_\theta^1\|_{L^p(\Omega)}\le \varepsilon.$
This proves \Cref{lem: univ.approx.flow}.
\end{proof}

\begin{remark}[On the reduction $m(\Omega)=\Omega$ a.e.]
If $m$ does not map $\Omega$ onto itself a.e., one may first modify $m$ on a set
of arbitrarily small measure to obtain $\tilde m$ with $\tilde m(\Omega)=\Omega$
a.e. The above construction then applies to
$\tilde m$. The incurred $L^p$-cost is $O(\eta^{1/p})$, which can be absorbed in
$\varepsilon$.
\end{remark}

\begin{lemma}\label{lem: mp.neural.ode.swap}
Fix $h>\delta>0$ and let $\Omega=[-L,L]^d$. For every $i\in h\mathbb Z^d\cap \Omega$ set $\square_{ih\delta}\coloneqq[i_1,i_1+h-\delta]\times\cdots\times[i_d,i_d+h-\delta]$. Then, for every $i,j\in h\mathbb Z^d$ there exists a measure-preserving flow map
$m:\R^d\to\R^d$ which is a solution of \eqref{eq: neural.ode} (hence bijective and
orientation-preserving) such that
\begin{align*}
&m(\square_{ih\delta})=\square_{jh\delta},\qquad
m(\square_{jh\delta})=\square_{ih\delta},\\
&m(\square_{kh\delta})=\square_{kh\delta}\quad\text{for every }k\in(h\mathbb Z^d\cap\Omega)\setminus\{i,j\}.
\end{align*}
Moreover, along the whole trajectory $t\mapsto x(t)$ of \eqref{eq: neural.ode} we have $\nabla\cdot v(x(t),\theta(t))=0$.
\end{lemma}

\begin{figure}[h!]
    \centering
    \[
    \includegraphics[scale=0.7]{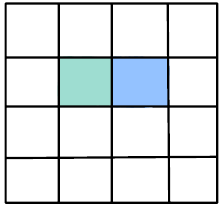}
    \quad \raisebox{1cm}{\(\xmapsto{\hspace{0.2cm}m\hspace{0.2cm}}\)} \quad\includegraphics[scale=0.46]{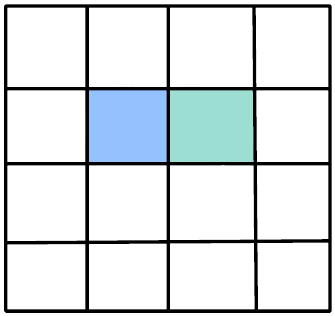}
    \]
    \caption{\Cref{lem: mp.neural.ode.swap}: the measure preserving map $m$ swaps the colored squares and leaves the whites invariant.}
    \label{fig:placeholder}
\end{figure}

\begin{proof}
It suffices to construct a divergence-free flow that swaps two adjacent
hypercubes; arbitrary pairs can then be swapped via a finite composition of
adjacent transpositions along a grid path. Since $h,\delta$ are fixed, we write
$\square_i\coloneqq\square_{ih\delta}$. Without loss of generality we treat the case
$j=i+he_1$.

We will build
\[
m=\upphi^{-1}\circ\uppsi\circ\upphi,
\]
where $\upphi$ ``isolates'' the two cubes and $\uppsi$ swaps them. Both $\upphi$ and
$\uppsi$ will be time-$T$ maps of \eqref{eq: neural.ode} with $\nabla_x\cdot v(\cdot,\theta(t))=0$
at all times; hence $m$ enjoys the same property.

\medskip
\noindent\textbf{Step 1: Isolating $\square_i$ and $\square_{i+he_1}$.}

\begin{claim}\label{cl: isolation}
There exists a measure-preserving and invertible map
$\upphi:\R^d\to\R^d$, which is a time-$T$ flow map of \eqref{eq: neural.ode} corresponding to a piecewise constant $\theta$ with finitely many switches satisfying
$\nabla_x\cdot v(\cdot,\theta(t))=0$ for all $t$, such that
\begin{align*}
\upphi(x)=x &\quad\text{ for  }x\in\square_i\cup\square_{i+he_1},\\
e_2\cdot\upphi(x)\le i_2-\delta &\quad\text{ for  }x\in\square_k,\ k\neq i,i+he_1.
\end{align*}
\end{claim}

\begin{figure}[h!]
    \centering
    \[
\begin{tikzcd}[column sep=large, row sep=large]
\includegraphics[scale=0.55]{Figures/perm-1.pdf}
  \quad\arrow[r, mapsto, "1"]\quad
&
\includegraphics[scale=0.55]{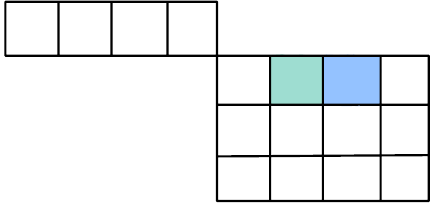}
  \arrow[d, mapsto, "2"]
\\
\includegraphics[scale=0.55]{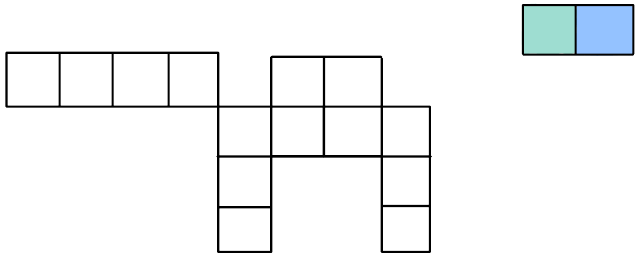}
&
\includegraphics[scale=0.55]{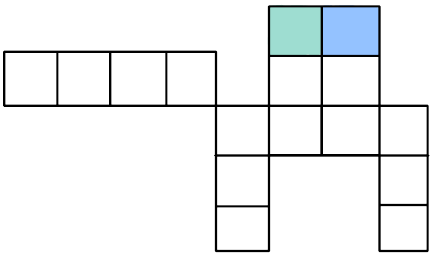}
  \quad \arrow[l, mapsto, "3"]\quad 
\end{tikzcd}
\]
    \caption{An illustration of \Cref{cl: isolation}.}
    \label{fig:placeholder}
\end{figure}

\begin{proof}[Proof of \Cref{cl: isolation}]
We employ Lemma~\ref{lem: several.d}: for a
unit vector $e_\ell$, thresholds $a<b$ with $b-a\ge\delta$, and a vector
$\uptau e_r$ with $r\neq \ell$, there exists a divergence-free flow whose time--$1$
map equals the translation $x\mapsto x+\uptau e_r$ on $\{x_\ell\ge b\}$, equals the
identity on $\{x_\ell\le a\}$, and is measure-preserving (and smooth) in the strip
$a<x_\ell<b$.

Define $\upphi$ as the composition
\[
\upphi=(\phi_{d,2}\circ\phi_{d,1})\circ(\phi_{d-1,2}\circ\phi_{d-1,1})\circ\cdots
\circ(\phi_{3,2}\circ\phi_{3,1})\circ(\phi_{1,2}\circ\phi_{1,1})\circ\phi_2,
\]
where:

\smallskip
\emph{(a) First, separate the layers above $i_2$.} Set
\[
\phi_2(x)=
\begin{cases}
x-2h\,e_1, & e_2\cdot x\ge i_2+h,\\
x, & e_2\cdot x\le i_2+h-\delta.
\end{cases}
\]
Then $\phi_2(\square_k)=\square_{k-2he_1}$ if $k_2>i_2$, and $\phi_2$ is the
identity on $\{e_2\cdot x\le i_2+h-\delta\}$, in particular on
$\square_i\cup\square_{i+he_1}$.

\smallskip
\emph{(b) Next, push down by $2h$ every cube that differs from $i$ in any
coordinate $m\neq 2$, while keeping $\square_i$ and $\square_{i+he_1}$ fixed.}
For each $m\in\{3,\dots,d\}$ set
\begin{align*}
\phi_{m,1}(x)&=
\begin{cases}
x-2h\,e_2, & e_m\cdot x\ge i_m+h,\\
x, & e_m\cdot x\le i_m+h-\delta,
\end{cases}\\
\phi_{m,2}(x)&=
\begin{cases}
x-2h\,e_2, & e_m\cdot x\le i_m-\delta,\\
x, & e_m\cdot x\ge i_m.
\end{cases}
\end{align*}
Since points in $\square_i\cup\square_{i+he_1}$ satisfy
$e_m\cdot x\in[i_m,i_m+h-\delta]$, neither $\phi_{m,1}$ nor
$\phi_{m,2}$ acts on them.

For $m=1$ we use asymmetric thresholds so as not to affect $\square_{i+he_1}$:
\begin{align*}
\phi_{1,1}(x)&=
\begin{cases}
x-2h\,e_2, & e_1\cdot x\ge i_1+2h,\\
x, & e_1\cdot x \le i_1+2h-\delta,
\end{cases} \\
\phi_{1,2}(x)&=
\begin{cases}
x-2h\,e_2, & e_1\cdot x\le i_1-\delta,\\
x, & e_1\cdot x \ge i_1.
\end{cases}
\end{align*}
Again, $\phi_{1,1}$ and $\phi_{1,2}$ are the identity on
$\square_i\cup\square_{i+he_1}$ because their $x_1$-coordinates lie in
$[i_1,i_1+2h-\delta]$.

\smallskip
The maps $\phi_{m,1},\phi_{m,2}$ for different $m\neq 2$ commute (they translate
in $e_2$ with selection depending only on $x_m$), although they need not commute
with $\phi_2$ (which selects by $x_2$). By inspection of the cases:

\begin{itemize}
\item If $k_2=i_2$ and $k\neq i,i+he_1$, then there exists $m\neq 2$ with
$k_m\neq i_m$. Hence either $\phi_{m,1}$ (if $k_m>i_m$) or $\phi_{m,2}$
(if $k_m<i_m$) moves $\square_k$ to $\square_{k-2he_2}$, which lies entirely in
$\{x_2\le i_2-h-\delta\}\subset\{x_2\le i_2-\delta\}$.

\item If $k_2>i_2$, then $\phi_2$ first sends $\square_k$ to the left,
$\square_{k-2he_1}$. Afterwards, $\phi_{1,2}$ applies (since all points have
$x_1\le i_1-\delta$), shifting the cube to $\square_{k-2he_1-2he_2}$, which
again lies in $\{x_2\le i_2-\delta\}$.
\end{itemize}
Thus $\upphi$ fixes $\square_i\cup\square_{i+he_1}$ and maps every other
$\square_k$ into the half-space $\{x_2\le i_2-\delta\}$, as claimed. Each block
map above is realized by Lemma~\ref{lem: several.d}, hence is a time--$1$ map of
a divergence-free flow; their composition $\upphi$ has the same properties.
\end{proof}

\medskip
\noindent\textbf{Step 2: Swapping $\upphi(\square_i)$ and $\upphi(\square_{i+he_1})$.}

\begin{claim}\label{cl: swap}
Under the conclusions of \Cref{cl: isolation}, there exists a measure-preserving and invertible map
$\upphi:\R^d\to\R^d$, which is the time-$T$ flow map of \eqref{eq: neural.ode} corresponding to a piecewise constant $\theta$ with finitely many switches satisfying
$\nabla_x\cdot v(\cdot,\theta(t))=0$ for all $t$, such that such that
\[
\uppsi(\square_i)=\square_{i+he_1},\qquad
\uppsi(\square_{i+he_1})=\square_i,\qquad
\uppsi|_{\{x_2\le i_2-\delta\}}=\mathsf{id}.
\]
Consequently,
\[
(\uppsi\circ\upphi)(\square_i)=\square_{i+he_1},\quad
(\uppsi\circ\upphi)(\square_{i+he_1})=\square_i,\quad
(\uppsi\circ\upphi)(\square_k)=\upphi(\square_k)
\]
for $k\neq i,i+he_1$.
\end{claim}

\begin{figure}[h!]
    \centering
    \[
\begin{tikzcd}[column sep=large, row sep=large]
\includegraphics[scale=0.55]{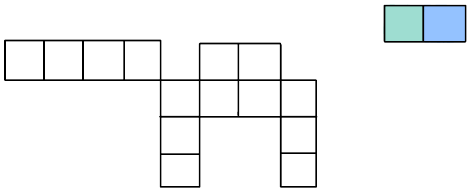}
  \quad\arrow[r, mapsto, "1"]\quad
&
\includegraphics[scale=0.55]{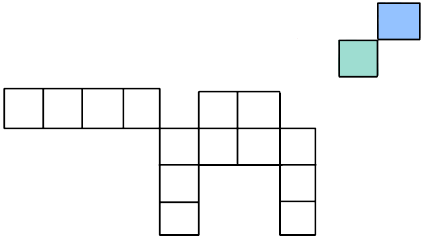}
  \arrow[d, mapsto, "2"]
\\
\includegraphics[scale=0.55]{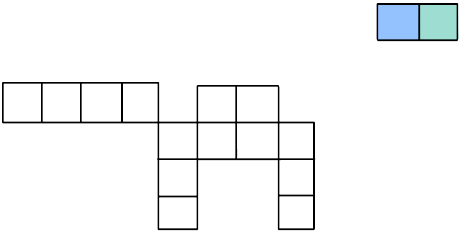}
&
\includegraphics[scale=0.55]{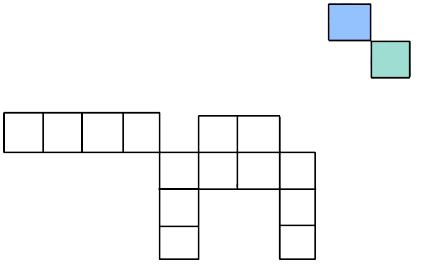}
  \quad \arrow[l, mapsto, "3"]\quad 
\end{tikzcd}
\]
    \caption{An illustration of \Cref{cl: swap}.}
    \label{fig:placeholder}
\end{figure}

\begin{proof}[Proof of \Cref{cl: swap}]
Set $\uptau\coloneqq h$. Using again Lemma~\ref{lem: several.d}, define
\[
\uppsi\coloneqq\varphi_4\circ\varphi_3\circ\varphi_2\circ\varphi_1,
\]
where each $\varphi_\ell$ is the time--$1$ map of a divergence-free vector field
and is the identity on $\{x_2\le i_2-\delta\}$. Concretely:
\begin{align}
\varphi_1(x)&=
\begin{cases}
x+\uptau e_1, & x_2\ge i_2,\\
x, & x_2\le i_2-\delta,
\end{cases}
\label{eq:psi1}\\[0.2em]
\varphi_2(x)&=
\begin{cases}
x+he_2, & x_2\ge i_2\ \text{ and }\ x_1\ge i_1+2h,\\
x, & x_2\le i_2-\delta\ \text{ or }\ x_1\le i_1+2h-\delta,
\end{cases}
\label{eq:psi2}\\[0.2em]
\varphi_3(x)&=
\begin{cases}
x-2h e_1, & x_2\ge i_2+h,\\
x, & x_2\le i_2+h-\delta,
\end{cases}
\label{eq:psi3}\\[0.2em]
\varphi_4(x)&=
\begin{cases}
x-he_2, & x_2\ge i_2+h,\\
x, & x_2\le i_2+h-\delta .
\end{cases}
\label{eq:psi4}
\end{align}
Let us track the two relevant cubes in the $(x_1,x_2)$-plane (all other
coordinates are unchanged):
\begin{align*}
A&\coloneqq\square_i=[i_1,i_1+h-\delta]\times[i_2,i_2+h-\delta],\\
B&\coloneqq\square_{i+he_1}=[i_1+h,i_1+2h-\delta]\times[i_2,i_2+h-\delta].
\end{align*}
Applying \eqref{eq:psi1}--\eqref{eq:psi4} in order:
\begin{itemize}
\item After $\varphi_1$: both $A$ and $B$ shift right by $h$:
\[
A_1=[i_1+h,i_1+2h-\delta]\times[i_2,i_2+h-\delta],\ 
B_1=[i_1+2h,i_1+3h-\delta]\times[i_2,i_2+h-\delta].
\]
\item After $\varphi_2$: only the far-right strip ($x_1\ge i_1+2h$) is lifted:
\[
A_2=A_1,\qquad
B_2=[i_1+2h,i_1+3h-\delta]\times[i_2+h,i_2+2h-\delta].
\]
\item After $\varphi_3$: only the lifted block ($x_2\ge i_2+h$) slides left by $2h$:
\[
A_3=A_2,\qquad
B_3=[i_1,i_1+h-\delta]\times[i_2+h,i_2+2h-\delta].
\]
\item After $\varphi_4$: lower the elevated block by $h$:
\begin{align*}
A_4&=A_3=[i_1+h,i_1+2h-\delta]\times[i_2,i_2+h-\delta]=\square_{i+he_1},\\
B_4&=[i_1,i_1+h-\delta]\times[i_2,i_2+h-\delta]=\square_i.
\end{align*}
\end{itemize}
Thus $\uppsi(\square_i)=\square_{i+he_1}$ and $\uppsi(\square_{i+he_1})=\square_i$.
Moreover each $\varphi_\ell$ is the identity on $\{x_2\le i_2-\delta\}$, hence
$\uppsi=\Id$ there. Since $\upphi(\square_k)\subset\{x_2\le i_2-\delta\}$ for all
$k\neq i,i+he_1$, we get $(\uppsi\circ\upphi)(x)=\upphi(x)$ for $x\in\square_k$,
$k\neq i,i+he_1$. Finally, every $\varphi_\ell$ is a time--$1$ map of a
divergence-free flow, so $\uppsi$ is measure-preserving and
orientation-preserving. This proves the claim.
\end{proof}

\medskip
\noindent\textbf{Step 3: Conclusion.}
Set
\[
m\coloneqq\upphi^{-1}\circ\uppsi\circ\upphi .
\]
Then
\[
m(\square_i)=\square_{i+he_1},\qquad
m(\square_{i+he_1})=\square_i,\qquad
m(\square_k)=\square_k
\]
for all $k\neq i,i+he_1$.
To swap arbitrary $\square_{i}$ and $\square_{j}$, compose finitely many such
adjacent swaps along a grid path from $i$ to $j$. Because $\upphi$ and $\uppsi$ (and
hence $m$) arise by concatenating time intervals on which
$\nabla_x\cdot v(\cdot,\theta(t))=0$, the Jacobian determinant along the flow solves
\[
\begin{cases}
\frac{\diff}{\diff t}\det \nabla X_t(x)=(\nabla_x\cdot  v(X_t(x),\theta(t)))\det \nabla X_t(x)=0\\
\det \nabla X_0\equiv 1
\end{cases}
\]
so $\det \nabla X_t\equiv 1>0$ for all $t$, and $m$ is measure-preserving and
orientation-preserving. This completes the proof.
\end{proof}

\subsubsection*{Compressible part}

\begin{lemma} \label{lem: exact.linear.nonuniform}
    Suppose $\Omega'\subset\R^{d-1}$ is a bounded domain, and consider the function $\upxi:[0, 1]\times \Omega'\to\R^d$ defined as
    \begin{equation*}
        \upxi(x) \coloneqq \upzeta(x_1)e_1+\sum_{j=2}^d x_je_j,
    \end{equation*}
    where $\upzeta:[0,1]\to\R$ is a continuous, piecewise affine, increasing function:
    \begin{equation*} \label{eq: exact.form.nonuniform}
        \upzeta(x_1)=\sum_{i=0}^n (\alpha_i x_1 + \beta_i)\mathbf{1}_{[y_i,y_{i+1})}(x_1),
    \end{equation*}
    where $0 = y_0 < y_1 < \dots < y_{n+1} = 1$. Then there exists a flow $\upxi_\theta^1:\R^d\to\R^d$ of \eqref{eq: neural.ode}, stemming from a piecewise constant $\theta$ with $n$ switches, such that
    \[
        \upxi \equiv \upxi_\theta^1 \quad \text{on } [0,1]\times\Omega'.
    \]
\end{lemma}

\begin{proof}
It is readily seen, by choosing $a(t)\in\R^d$ and $w(t)\in\R^d$ to only have non-zero values in the first coordinate,
that it suffices to construct a scalar flow $\upzeta_\theta^t:\R\to\R$ of \eqref{eq: neural.ode} such that
$\upzeta_\theta^1\equiv \upzeta$ on $[0,1]$.

Write $\upzeta$ as in \eqref{eq: exact.form.nonuniform} with breakpoints
$0=y_0<y_1<\dots<y_{n+1}=1$ and slopes $\alpha_i>0$.
Continuity of $\upzeta$ at $y_i$ is equivalent to
\[
\alpha_{i-1}y_i+\beta_{i-1}=\alpha_i y_i+\beta_i,\qquad i=1,\dots,n,
\]
hence
\[
\beta_i=\beta_{i-1}+(\alpha_{i-1}-\alpha_i)y_i.
\]

\medskip
\noindent\textbf{Step 1: Elementary maps.}
Set $h_i\coloneqq y_{i+1}-y_i$.
For each $i=0,\dots,n$, we will build a map $S_i:\R\to\R$ which is the time-$h_i$ flow map of a scalar
\eqref{eq: neural.ode} and satisfies:
\begin{align*}
S_i(x)=x &\text{ for }x\le \Phi_{i-1}(y_i),\\
S_i(x)=\Phi_{i-1}(y_i) + \frac{\alpha_i}{\alpha_{i-1}}\left(x-\Phi_{i-1}(y_i)\right) &\text{ for }x\ge \Phi_{i-1}(y_i),
\end{align*}
where by convention $\alpha_{-1}=1$ and $\Phi_{-1}=\Id$.
Define the partial compositions
\(\Phi_i \coloneqq S_i\circ S_{i-1}\circ\cdots\circ S_0\)
for \( i=0,\dots,n.\)
Fix $i\ge 0$ and set the threshold $c_i\coloneqq \Phi_{i-1}(y_i)$.
Consider the scalar ODE on the time interval of length $h_i$:
\begin{equation}\label{eq:slope-change-ODE}
\dot x(t)=\gamma_i\left(x(t)-c_i\right)_+,
\qquad
\gamma_i\coloneqq \frac{1}{h_i}\log\left(\frac{\alpha_i}{\alpha_{i-1}}\right),
\end{equation}
with initial condition $x(0)=x_0$.
By the explicit one-dimensional formula in Lemma~\ref{lem: 1d}(1) (with $a\equiv 1$ and bias $b=c_i$),
the time-$h_i$ map of \eqref{eq:slope-change-ODE} is exactly
\[
S_i(x_0)=
\begin{cases}
x_0, & x_0\le c_i,\\[0.2em]
c_i + e^{\gamma_i h_i}(x_0-c_i)
= c_i + \dfrac{\alpha_i}{\alpha_{i-1}}(x_0-c_i), & x_0\ge c_i,
\end{cases}
\]
which is the desired slope-change at the threshold $c_i$.

\medskip
\noindent\textbf{Step 2: Induction.}
Assume first $\beta_0=0$.
We prove by induction that for each $i=0,\dots,n$,
\begin{equation}\label{eq:Phi-induction}
\Phi_i(x)=
\alpha_i x + \sum_{j=1}^{i}(\alpha_{j-1}-\alpha_j)y_j\,\mathbf 1_{(y_j,1]}(x)
\qquad\text{for all }x\in[0,1].
\end{equation}
For $i=0$, we have $c_0=\Phi_{-1}(y_0)=y_0=0$, hence
$S_0(x)=e^{\gamma_0 h_0}x=\alpha_0 x$ for $x\ge 0$, so $\Phi_0(x)=\alpha_0 x$ and \eqref{eq:Phi-induction} holds.

Assume \eqref{eq:Phi-induction} holds for $i-1$.
Then $\Phi_{i-1}$ is continuous and strictly increasing on $[0,1]$, and $c_i=\Phi_{i-1}(y_i)$.
If $x\le y_i$ then $\Phi_{i-1}(x)\le \Phi_{i-1}(y_i)=c_i$, so $S_i(\Phi_{i-1}(x))=\Phi_{i-1}(x)$ and hence
$\Phi_i(x)=\Phi_{i-1}(x)$.

If $x>y_i$, then $\Phi_{i-1}(x)\ge c_i$ and thus
\[
\Phi_i(x)
= S_i(\Phi_{i-1}(x))
= c_i + \frac{\alpha_i}{\alpha_{i-1}}\left(\Phi_{i-1}(x)-c_i\right).
\]
Using the induction hypothesis at $x$ and at $y_i$ gives
\[
\Phi_{i-1}(x)-\Phi_{i-1}(y_i)=\alpha_{i-1}(x-y_i),
\]
because the jump terms $\sum_{j=1}^{i-1}(\alpha_{j-1}-\alpha_j)y_j\,\mathbf 1_{(y_j,1]}(\cdot)$ agree for $x$ and $y_i$.
Therefore
\[
\Phi_i(x)=\Phi_{i-1}(y_i)+\frac{\alpha_i}{\alpha_{i-1}}\cdot \alpha_{i-1}(x-y_i)
=\Phi_{i-1}(y_i)+\alpha_i(x-y_i).
\]
Finally, by \eqref{eq:Phi-induction} at $y_i$,
\[
\Phi_{i-1}(y_i)=\alpha_{i-1}y_i+\sum_{j=1}^{i-1}(\alpha_{j-1}-\alpha_j)y_j,
\]
hence
\[
\Phi_i(x)=\alpha_i x + \sum_{j=1}^{i}(\alpha_{j-1}-\alpha_j)y_j,
\qquad x>y_i,
\]
which is exactly \eqref{eq:Phi-induction} for $i$.
Thus \eqref{eq:Phi-induction} holds for all $i$, and in particular $\Phi_n=\upzeta$ on $[0,1]$ when $\beta_0=0$.

\medskip
\noindent\textbf{Step 3: Adding the initial translation $\beta_0\neq 0$.}
If $\beta_0\neq 0$, prepend a translation map $T_{\beta_0}(x)=x+\beta_0$, realized with two switches by
Lemma~\ref{lem: 1d}(2). Apply the above construction to $\upzeta-\beta_0$ (which has the same slopes and breakpoints and satisfies
$(\upzeta-\beta_0)(0)=0$). Composing $T_{\beta_0}$ with the $n$ slope-change stages yields a piecewise-constant control
with $n+2$ switches whose time-$1$ flow equals $\upzeta$ on $[0,1]$. Finally, extending componentwise gives the desired map.
\end{proof}

\subsection{Stability estimates}\label{subsec: stab.estimates}

The goal of this section is to two elementary stability estimates.

\begin{lemma} \label{lem: contraction}
    Let $\Omega\subseteq\R^d$,  $\mu_1,\mu_2\in \mathscr{P}(\R^d)$, and let $\phi:\R^d\to\mathbb{R}^d$ be measurable. 
    Then
    \begin{equation*}
        \left|\phi_\#\mu_1-\phi_\#\mu_2\right|_{\mathsf{TV}(\Omega)}\leq \left|\mu_1-\mu_2\right|_{\mathsf{TV}(\phi^{-1}(\Omega))}.
    \end{equation*}
\end{lemma}

\begin{proof}
We have
   \begin{align*}
       \left|\phi_\#\mu_1-\phi_\#\mu_2\right|_{\mathsf{TV}(\Omega)}&=\sup_{\mathscr{A}\subset \Omega} \left|\mu_1(\phi^{-1}(\mathscr{A}))-\mu_2(\phi^{-1}(\mathscr{A}))\right|\\
       &\leq\sup_{\mathscr{B}\subset \phi^{-1}(\Omega)} \left|\mu_1(\mathscr{B})-\mu_2(\mathscr{B})\right|\\
       &=\left|\mu_1-\mu_2\right|_{\mathsf{TV}(\phi^{-1}(\Omega))}.
   \end{align*}
   where the inequality stems from the fact that $\phi$ may not be surjective. Should $\phi$ be surjective, then the inequality is an equality.
\end{proof}

In the context in which $\phi$ is the flow of $\dot{x}(t) = v(t,x(t))$ with $x(0)=x_0$ say, \Cref{lem: contraction} is simply restating the fact that the semigroup for the continuity equation $\partial_t \rho(t,x)+\nabla\cdot (v(t,x)\rho(t,x))=0$ is a contraction on $L^1(\Omega)$ for any $\Omega\subseteq\R^d$.

We have the following estimate for pushforwards of a measure via measure-preserving maps.

\begin{lemma} \label{thm: stability.measure.preserving}
Suppose $\rho\in L^\infty(\R^d)$ a probability density of the form
\[
\rho=\sum_{k=1}^K \rho_k\mathbf 1_{E_k},
\]
where $K\in\N$, $E_k\subset\R^d$ are pairwise disjoint, bounded Lipschitz domains, and
$\rho_k\in C^{0,1}(E_k)$. Denote $\omega\coloneqq \mathrm{supp}\,\rho=\bigcup_{k} E_k$.
Let $m_1:\R^d\to\R^d$ be a measure--preserving bijection and let
$m_2:\R^d\to\R^d$ be a measure--preserving bi--Lipschitz homeomorphism. 
Then for every $p\in\mathbb{N}$ there exists a constant $C=C(d,p,\rho,\omega,m_2)>0$ such that
\begin{equation*} \label{eq:BV-piecewise-stability}
|m_{1\#}\mu-m_{2\#}\mu|_{\mathsf{TV}(\R^d)}
\le C\|m_1-m_2\|_{L^p(\omega)}^{\frac{p}{p+1}}.
\end{equation*}
\end{lemma}

\begin{proof}
Define $\eta\coloneqq m_2^{-1}\circ m_1$. Since both $m_1$ and $m_2$ are
measure-preserving bijections, so is $\eta$.
Moreover, since total variation is invariant under pushforward by a bijection,
\[
|m_{1\#}\mu-m_{2\#}\mu|_{\mathsf{TV}}
=
| (m_2^{-1})_\#(m_{1\#}\mu-m_{2\#}\mu)|_{\mathsf{TV}}
=
|\eta_\#\mu-\mu|_{\mathsf{TV}}.
\]
Because $\eta$ preserves Lebesgue measure (i.e.\ $|\det\nabla\eta|=1$ a.e.) and is bijective,
the pushforward $\eta_\#\mu$ has density $\rho\circ\eta^{-1}$, hence
\[
|\eta_\#\mu-\mu|_{\mathsf{TV}}
=
\int_{\R^d}|\rho(x)-\rho(\eta^{-1}(x))|\diff x.
\]
Changing variables $x=\eta(y)$ and using measure preservation gives
\begin{equation}\label{eq:TV-eta}
|\eta_\#\mu-\mu|_{\mathsf{TV}}
=
\int_{\R^d}|\rho(\eta(y))-\rho(y)|\diff y.
\end{equation}
We now estimate the right-hand side.
Let $\Sigma\coloneqq \bigcup_{k=1}^K \partial E_k$ and set
\[
\ell_*\coloneqq \max_{1\le k\le K}\|\rho_k\|_{C^{0,1}(E_k)}.
\]
Fix $\delta>0$ and define the sets
\begin{align*}
X_\delta&\coloneqq \bigcup_{k=1}^K\Big\{y\in E_k:\dist(y,\partial E_k)>\delta,\ |\eta(y)-y|\le\delta\Big\},\\
Y_\delta&\coloneqq \{y\in\R^d:\dist(y,\Sigma)\le\delta\},\\
Z_\delta&\coloneqq \{y\in\omega:|\eta(y)-y|>\delta\}.
\end{align*}

\medskip
\noindent\textbf{Step 1: Lipschitz control on $X_\delta$.}
If $y\in X_\delta\cap E_k$, then $|\eta(y)-y|\le\delta$ and $\dist(y,\partial E_k)>\delta$,
so necessarily $\eta(y)\in E_k$ as well. Therefore
\[
|\rho(\eta(y))-\rho(y)|=|\rho_k(\eta(y))-\rho_k(y)|
\le \|\rho_k\|_{C^{0,1}}\,|\eta(y)-y|
\le \ell_*\,\delta.
\]
Integrating yields
\begin{equation}\label{eq:Xdelta}
\int_{X_\delta}|\rho(\eta)-\rho|
\le \ell_*|\omega|\,\delta.
\end{equation}

\medskip
\noindent\textbf{Step 2: Control on the complement of $X_\delta$ inside $\omega$.}
On $\omega$ we have the decomposition $\omega = X_\delta \cup (\omega\cap Y_\delta)\cup Z_\delta$,
because if $y\in\omega$ is not in $X_\delta$, then either $\dist(y,\Sigma)\le\delta$ (hence $y\in Y_\delta$)
or $|\eta(y)-y|>\delta$ (hence $y\in Z_\delta$).
Using the crude bound $|\rho(\eta)-\rho|\le 2\|\rho\|_{L^\infty}$, we obtain
\begin{equation}\label{eq:omega-complement}
\int_{\omega\setminus X_\delta}|\rho(\eta)-\rho|
\le 2\|\rho\|_{L^\infty}\left(|\omega\cap Y_\delta|+|Z_\delta|\right).
\end{equation}

\medskip
\noindent\textbf{Step 3: Outside $\omega$.}
The integral \eqref{eq:TV-eta} also includes $y\in\omega^c$, where $\rho(y)=0$, so
\[
\int_{\omega^c}|\rho(\eta(y))-\rho(y)|\diff y
=
\int_{\omega^c}\rho(\eta(y))\diff y
\le \|\rho\|_{L^\infty}\,|\omega^c\cap \eta^{-1}(\omega)|.
\]
Observe that if $y\in\omega$ and $\dist(y,\Sigma)>\delta$ and $|\eta(y)-y|\le\delta$, then $\eta(y)\in\omega$.
Hence $\{y\in\omega:\eta(y)\notin\omega\}\subset (\omega\cap Y_\delta)\cup Z_\delta$.
Since $\eta$ is measure-preserving, the entering and leaving sets have the same measure:
\[
|\omega^c\cap\eta^{-1}(\omega)|=|\omega\cap\eta^{-1}(\omega^c)|
\le |\omega\cap Y_\delta|+|Z_\delta|.
\]
Therefore,
\begin{equation}\label{eq:leakage}
\int_{\omega^c}\rho(\eta(y))\diff y
\le \|\rho\|_{L^\infty}\left(|\omega\cap Y_\delta|+|Z_\delta|\right).
\end{equation}

\medskip
\noindent\textbf{Step 4: Estimate of $Y_\delta$ and $Z_\delta$.}
Combining \eqref{eq:Xdelta}, \eqref{eq:omega-complement}, and \eqref{eq:leakage} with
\eqref{eq:TV-eta} yields
\[
|\eta_\#\mu-\mu|_{\mathsf{TV}}
\le
\ell_*|\omega|\,\delta
+ C\|\rho\|_{L^\infty}\,|Y_\delta|
+ C\|\rho\|_{L^\infty}\,|Z_\delta|.
\]
Since $\Sigma$ is a finite union of Lipschitz hypersurfaces, a tubular-neighborhood estimate gives
\[
|Y_\delta|\le C_d\,\mathrm{per}(\omega)\,\delta
\qquad\text{for all }\delta\in(0,\delta_0),
\]
for some $\delta_0>0$ and $C_d$ depending only on $d$ and the Lipschitz character.

Also, by Markov's inequality on $\omega$,
\[
|Z_\delta|
\le \delta^{-p}\|\eta-\Id\|_{L^p(\omega)}^p.
\]
Since $m_2^{-1}$ is Lipschitz,
\[
\|\eta-\Id\|_{L^p(\omega)}
=
\|m_2^{-1}\circ m_1-m_2^{-1}\circ m_2\|_{L^p(\omega)}
\le \|m_2^{-1}\|_{C^{0,1}}\|m_1-m_2\|_{L^p(\omega)}.
\]
Hence we obtain
\[
|\eta_\#\mu-\mu|_{\mathsf{TV}}
\le A\,\delta
+ B\,\delta^{-p}\|m_1-m_2\|_{L^p(\omega)}^p,
\]
for constants $A,B$ depending on $(d,p,\rho,\omega,m_2)$.

Optimizing in $\delta$ (take $\delta=(pB/A)^{1/(p+1)}\|m_1-m_2\|_{L^p(\omega)}^{p/(p+1)}$)
yields
\[
|\eta_\#\mu-\mu|_{\mathsf{TV}}
\le C\|m_1-m_2\|_{L^p(\omega)}^{\frac{p}{p+1}},
\]
which, together with $|m_{1\#}\mu-m_{2\#}\mu|_{\mathsf{TV}}=|\eta_\#\mu-\mu|_{\mathsf{TV}}$,
concludes the proof.
\end{proof}

\subsection{Proof of \Cref{thm: gen.thm}}
\label{sec: proof.1}

\begin{proof}[Proof of \Cref{thm: gen.thm}]
We split the proof in two parts.

\subsubsection*{Part 1. The \(L^p\)–estimate.}
Let \(\phi_\varepsilon\) be the Lagrange interpolant from \Cref{thm: lagrange.approx}, which by \Cref{prop: lagrange.decomposition}, admits a factorization
\[
\phi_\varepsilon = m_{2\varepsilon}\circ g_\varepsilon\circ m_{1\varepsilon},
\]
where \(m_{1\varepsilon},m_{2\varepsilon}\) are measure–preserving bijections and \(g_\varepsilon\) is a $C^1$ diffeomorphism. 
We construct the controlled flow \(\phi_\theta\) by approximating the three factors with flows of \eqref{eq: neural.ode}:
\[
\phi_\theta = m_{2\theta}\circ g_\theta\circ m_{1\theta}.
\]
Set \(f_\varepsilon\coloneqq g_\varepsilon\circ m_{1\varepsilon}\) and \(f_\theta\coloneqq g_\theta\circ m_{1\theta}\). Then
\begin{align}\label{eq:first-Lp}
\|\phi_\varepsilon-\phi_\theta\|_{L^p(\Omega)}
&\le\|m_{2\varepsilon}\circ f_\varepsilon-m_{2\theta}\circ f_\varepsilon\|_{L^p(\Omega)}
   + \|m_{2\theta}\circ f_\varepsilon-m_{2\theta}\circ f_\theta\|_{L^p(\Omega)}.
\end{align}
For the first term, since $m_{1\varepsilon}$ is measure-preserving,
\begin{align*}
\int_\Omega |m_{2\varepsilon}(f_\varepsilon(x))-m_{2\theta}(f_\varepsilon(x))|^p\diff x
&= \int_{m_{1\varepsilon}(\Omega)} |m_{2\varepsilon}(g_\varepsilon(y))-m_{2\theta}(g_\varepsilon(y))|^p\diff y\\
&\hspace{-1cm}= \int_{g_\varepsilon(m_{1\varepsilon}(\Omega))} |m_{2\varepsilon}(z)-m_{2\theta}(z)|^p|\det\nabla g_\varepsilon^{-1}(z)|\diff z,
\end{align*}
and therefore
\begin{equation*}\label{eq:area-1p}
\|m_{2\varepsilon}\circ f_\varepsilon - m_{2\theta}\circ f_\varepsilon\|_{L^p(\Omega)}
\le \|\det\nabla g_\varepsilon^{-1}\|_{L^\infty}^{\frac1p}
   \,\|m_{2\varepsilon}-m_{2\theta}\|_{L^p(g_\varepsilon(m_{1\varepsilon}(\Omega)))}.
\end{equation*}
For the second term in \eqref{eq:first-Lp}, as \(m_{2\theta}\) is Lipschitz,
\[
\|m_{2\theta}\circ f_\varepsilon-m_{2\theta}\circ f_\theta\|_{L^p(\Omega)}
\le \|m_{2\theta}\|_{C^{0,1}}\|f_\varepsilon-f_\theta\|_{L^p(\Omega)}.
\]
Then
\begin{align*}
\|f_\varepsilon-f_\theta\|_{L^p(\Omega)}
&=\|g_\varepsilon\circ m_{1\varepsilon}-g_\theta\circ m_{1\theta}\|_{L^p(\Omega)}\\
&\le \|g_\varepsilon-g_\theta\|_{L^p(m_{1\varepsilon}(\Omega))}
   +\|g_\theta\|_{C^{0,1}}\,\|m_{1\varepsilon}-m_{1\theta}\|_{L^p(\Omega)}.
\end{align*}
By \Cref{lem: univ.approx.flow}, we can choose \(m_{j\theta}\) to approximate \(m_{j\varepsilon}\) in \(L^p\) as well as desired (on a domain even larger than $\Omega$, to account for what is done in Part 2 as well), and by \Cref{lem: exact.linear.nonuniform}, we can construct \(g_\theta\) to be exactly equal to \(g_\varepsilon\) on $m_{1\varepsilon}(\Omega)$. Combining the last three displays with \eqref{eq:first-Lp} and adding the error \(\|\phi-\phi_\varepsilon\|_{L^p}\) from \Cref{thm: lagrange.approx}, we can arrange
\[
\|\phi-\phi_\theta\|_{L^p(\Omega)}\le \varepsilon
\]
by taking the three approximations sufficiently fine.

\subsubsection*{Part 2. The total variation estimate.}
We split
\begin{equation}\label{eq:TV-split}
|\phi_{\#}\mu-\phi_{\theta\#}\mu|_{\mathsf{TV}(\Omega)}
\le |\phi_{\#}\mu-\phi_{\varepsilon\#}\mu|_{\mathsf{TV}(\Omega)}
   +|\phi_{\varepsilon\#}\mu-\phi_{\theta\#}\mu|_{\mathsf{TV}(\Omega)}.
\end{equation}
We look at one term at a time.

\emph{(II.a) The term \(|\phi_{\#}\mu-\phi_{\varepsilon\#}\mu|\).}
Write \(x=\phi(y)\) and set \(\eta\coloneqq\phi_\varepsilon^{-1}\circ\phi\). By change of variables,
\[
|\phi_{\#}\mu-\phi_{\varepsilon\#}\mu|_{\mathsf{TV}(\Omega)}
= \int_{\phi^{-1}(\Omega)}| \rho(y) - \rho(\eta(y))||\det\nabla\eta(y)| \diff y,
\]
and hence
\[
|\phi_{\#}\mu-\phi_{\varepsilon\#}\mu|_{\mathsf{TV}(\Omega)}
\le \underbrace{\int_{\phi^{-1}(\Omega)} |\rho(\eta)-\rho|\diff y}_{\mathrm{(A)}} +\|\rho\|_{L^\infty}\underbrace{\int_{\phi^{-1}(\Omega)}||\det\nabla\eta|-1|\diff y}_{\mathrm{(B)}}.
\]
Since $\rho$ is Lipschitz, we have
\begin{align*} 
\mathrm{(A)}
=\int_{\phi^{-1}(\Omega)}|\rho(\eta)-\rho|\mathrm{d}y
&\le\|\rho\|_{C^{0,1}}\,\|\eta-\mathsf{id}\|_{L^1(\phi^{-1}(\Omega))}\nonumber\\
&\le\|\rho\|_{C^{0,1}}\,|\phi^{-1}(\Omega)|^{1-\frac1p}\,\|\eta-\mathsf{id}\|_{L^p(\phi^{-1}(\Omega))}.
\end{align*}
Moreover, since $\phi_\varepsilon^{-1}$ is Lipschitz on $\phi^{-1}(\Omega)$,
\begin{align}\label{eq:A-Lip-simple-2}
\|\eta-\mathsf{id}\|_{L^p(\phi^{-1}(\Omega))}
&=\|\phi_\varepsilon^{-1}\circ\phi-\phi_\varepsilon^{-1}\circ\phi_\varepsilon\|_{L^p(\phi^{-1}(\Omega))}\nonumber\\
&\le\|\phi_\varepsilon^{-1}\|_{C^{0,1}(\phi^{-1}(\Omega))}\|\phi-\phi_\varepsilon\|_{L^p(\phi^{-1}(\Omega))}.
\end{align}
For (B), by multilinearity of the determinant,
\[
||\det\nabla\eta|-1|
\le C_d(\|\nabla\eta\|^{d-1}+\|I_d\|^{d-1})|\nabla\eta-I_d|,
\]
and we also have the identity $\nabla\eta(y)-I_d
=(\nabla \phi_\varepsilon^{-1}(\phi(y))-\nabla \phi^{-1}(\phi(y)))\nabla \phi(y)$.
Therefore,
\begin{equation}\label{eq:grad-eta-L1}
\|\nabla\eta-I_d\|_{L^1(\phi^{-1}(\Omega))}
\le\|\nabla\phi\|_{L^\infty(\phi^{-1}(\Omega))}\,
        \|\det\nabla\phi^{-1}\|_{L^\infty(\Omega)}\,
        \|\nabla\phi_\varepsilon^{-1}-\nabla\phi^{-1}\|_{L^1(\Omega)}.
\end{equation}
By \Cref{thm: lagrange.approx}, the right–hand sides of \eqref{eq:A-Lip-simple-2}–\eqref{eq:grad-eta-L1} can be made arbitrarily small; we obtain
\[
|\phi_{\#}\mu-\phi_{\varepsilon\#}\mu|_{\mathsf{TV}(\Omega)}\le \frac{\varepsilon}{2}.
\]
As for the second term,
\begin{align*}
|\phi_{\varepsilon\#}\mu-\phi_{\theta\#}\mu|_{\mathsf{TV}(\Omega)}
&\le |(m_{2\varepsilon} \circ f_\varepsilon)_{\#}\mu-(m_{2\theta} \circ f_\varepsilon)_{\#}\mu|_{\mathsf{TV}(\Omega)} \\
&\quad + |(m_{2\theta} \circ f_\varepsilon)_{\#}\mu-(m_{2\theta} \circ f_\theta)_{\#}\mu|_{\mathsf{TV}(\Omega)} \\
&=: I_1 + I_2.
\end{align*}
By \Cref{lem: contraction} with \(\phi=m_{2\theta}\),
\[
I_2 \le\|f_{\varepsilon\#}\mu - f_{\theta\#}\mu|_{\mathsf{TV}(m_{2\theta}^{-1}(\Omega))}.
\]
By \Cref{lem: exact.linear.nonuniform}, we may choose the compressible factor so that 
\(
g_\theta\equiv g_\varepsilon\) on $m_{1\varepsilon}(\Omega)$,
hence \(f_\theta=g_\varepsilon\circ m_{1\theta}\). Applying \Cref{lem: contraction} again with \(\phi=g_\varepsilon\),
\[
I_2\le|m_{1\varepsilon\#}\mu - m_{1\theta\#}\mu|_{\mathsf{TV}(g_\varepsilon^{-1}(m_{2\theta}^{-1}(\Omega)))}.
\]
Set $\Omega_1\coloneqq g_\varepsilon^{-1}(m_{2\theta}^{-1}(\Omega))\cap \mathrm{supp}\, \rho$. Now \Cref{thm: stability.measure.preserving} applied to \(\rho_{|\Omega_1}\) yields, for some 
\(C>0\),
\begin{equation}\label{eq:I2-bound-final}
I_2\le C\|m_{1\varepsilon}-m_{1\theta}\|_{L^p(\Omega_1)}^{\frac{p}{p+1}}.
\end{equation}
Set \(\nu\coloneqq f_{\varepsilon\#}\mu\). Then
\[
I_1
= |(m_{2\varepsilon})_{\#}\nu-(m_{2\theta})_{\#}\nu|_{\mathsf{TV}(\Omega)}.
\]
Now apply \Cref{thm: stability.measure.preserving} to $\nu$ to obtain
\begin{equation}\label{eq:I1-bound-final}
I_1\le C'\|m_{2\varepsilon}-m_{2\theta}\|_{L^p(g_\varepsilon(m_{1\varepsilon}(\Omega)))}^{\frac{p}{p+1}}.
\end{equation}
Combining (II.a), \eqref{eq:I2-bound-final}–\eqref{eq:I1-bound-final}, and \eqref{eq:TV-split}, we obtain
\[
|\phi_{\#}\mu-\phi_{\theta\#}\mu|_{\mathsf{TV}(\Omega)}
\le\frac{\varepsilon}{2}
+C\,\|m_{1\varepsilon}-m_{1\theta}\|_{L^p}^{\frac{p}{p+1}}
+C'\|m_{2\varepsilon}-m_{2\theta}\|_{L^p}^{\frac{p}{p+1}}.
\]
Choosing $\varepsilon$ so that the two \(L^p\)–errors are sufficiently small proves \Cref{thm: gen.thm}.
\end{proof}

\begin{remark} We briefly explain why the measures to which we apply \Cref{thm: stability.measure.preserving} have densities that are piecewise
Lipschitz on finitely many Lipschitz domains.
By assumption $\rho\in L^\infty \cap C^{0,1}(\R^d)$.
Fix a bounded set $B\subset\R^d$ containing all domains that appear in Part~2 (e.g. a large box containing
$\Omega$, $m_{1\varepsilon}(\Omega)$, $g_\varepsilon(m_{1\varepsilon}(\Omega))$, and their small enlargements).
On $B$, the maps $m_{1\varepsilon}$ and $g_\varepsilon$ are piecewise affine with finitely many pieces:
$m_{1\varepsilon}$ is piecewise affine on the finite triangulation used in \Cref{prop: lagrange.decomposition},
and $g_\varepsilon=\nabla\varphi$ is affine on each interval where $\psi'$ is affine.
Therefore the composition $f_\varepsilon=g_\varepsilon\circ m_{1\varepsilon}$ is piecewise affine on a finite polyhedral partition of $B$
(obtained by taking a common refinement of the triangulation and the interval partition pulled back by $m_{1\varepsilon}$).

Set $\nu=f_{\varepsilon\#}\mu$.
Since $f_\varepsilon$ is a bi-Lipschitz piecewise affine map on each cell of the refined partition, it maps each cell
(a Lipschitz polytope) onto a Lipschitz polytope, and the change-of-variables formula shows that on the image of each cell,
the density of $\nu$ is given by a Lipschitz function composed with an affine map and multiplied by a constant Jacobian factor.
Hence $\nu$ admits a representation
\[
\frac{\diff \nu}{\diff x}=\sum_{j=1}^J \nu_j\,\mathbf 1_{F_j}
\]
where the $F_j$ are finitely many bounded Lipschitz (indeed polyhedral) domains, and each $\nu_j$ is Lipschitz on $F_j$.
In other words, $\nu$ satisfies the structural hypothesis of \Cref{thm: stability.measure.preserving}. 

The same reasoning applies after restricting to subsets such as
$\Omega_1=g_\varepsilon^{-1}(m_{2\theta}^{-1}(\Omega))\cap\supp\,\rho$:
since $m_{2\theta}$ is piecewise affine, the set $m_{2\theta}^{-1}(\Omega)$ is a finite union of polytopes,
and since $g_\varepsilon$ is piecewise affine in one coordinate, $g_\varepsilon^{-1}(m_{2\theta}^{-1}(\Omega))$ is again a finite union of Lipschitz  domains after refining the interval partition.
Intersecting with $\supp\,\rho$ preserves the ``finite union of Lipschitz domains'' property.
Therefore the restricted measures appearing in the bounds also satisfy the assumptions of \Cref{thm: stability.measure.preserving}.
\end{remark}

\section{Proof of \Cref{thm:sobolev_to_switches}}  \label{sec: proof.2}

\subsection{Preliminaries}

We first prove 

\begin{lemma}\label{lem:TV_stability_diffeo}
Let $\rho\in L^\infty\cap C^{0,1}(\R^d)$.
Let $\phi,\psi:\R^d\to\R^d$ be $C^1$ diffeomorphisms and set $\eta\coloneqq \phi^{-1}\circ\psi$.
Assume $\rho$ is supported in a bounded set $\omega$ and define
\[
M\coloneqq\|\log\det\nabla\eta\|_{L^\infty(\omega)}.
\]
Then
\begin{equation}\label{eq:TV_stability_diffeo}
|\phi_\#\mu-\psi_\#\mu|_{\mathsf{TV}(\R^d)}
\le
e^{M}\|\rho\|_{C^{0,1}}\|\eta-\mathsf{id}\|_{L^1}
+e^{M}\|\rho\|_{L^\infty}\|\log\det\nabla\eta\|_{L^1}
\end{equation}
where in the equation above $L^1=L^1(\omega)$.
\end{lemma}

\begin{proof}[Proof of \Cref{lem:TV_stability_diffeo}]
Since $\phi$ is a bijection, total variation is invariant under pushforward by $\phi^{-1}$:
\[
|\phi_\#\mu-\psi_\#\mu|_{\mathsf{TV}}
=
|(\phi^{-1})_\#(\phi_\#\mu-\psi_\#\mu)|_{\mathsf{TV}}
=
|\mu-\eta_\#\mu|_{\mathsf{TV}}.
\]
Because $\eta$ is a $C^1$ diffeomorphism, changing variables $y=\eta(x)$ yields the standard identity
\[
|\mu-\eta_\#\mu|_{\mathsf{TV}}
=
\int_{\R^d}\big|\rho(\eta(x))\det\nabla\eta(x)-\rho(x)\big|\diff x.
\]
Since $\rho$ is supported in $\omega$, the integral reduces to $\omega$ up to a null set.
Split
\[
|\rho(\eta)\det\nabla\eta-\rho|
\le
|\rho(\eta)-\rho|\,|\det\nabla\eta|
+\rho\,|\det\nabla\eta-1|.
\]
For the first term, $\rho$ is globally Lipschitz so
$|\rho(\eta(x))-\rho(x)|\le \|\rho\|_{C^{0,1}}\,|\eta(x)-x|$.
Also $|\det\nabla\eta|=\exp(\log\det\nabla\eta)\le e^{M}$ on $\omega$.
Therefore
\[
\int_\omega |\rho(\eta)-\rho|\,|\det\nabla\eta|
\le
e^{M}\,\|\rho\|_{C^{0,1}}\,\|\eta-\mathsf{id}\|_{L^1(\omega)}.
\]
For the second term, use $|e^a-1|\le e^{|a|}|a|$ with $a=\log\det\nabla\eta$ to get
$|\det\nabla\eta-1|\le e^{M}|\log\det\nabla\eta|$ on $\omega$.
Hence
\[
\int_\omega \rho\,|\det\nabla\eta-1|
\le
\|\rho\|_{L^\infty}\,e^{M}\,\|\log\det\nabla\eta\|_{L^1(\omega)}.
\]
Combining proves \eqref{eq:TV_stability_diffeo}.
\end{proof}

Henceforth, $\mathscr{F}$ denotes the Fourier transform.
Let $\Omega\subset\R^d$ be bounded and set $\langle \xi \rangle \coloneqq \sup_{x\in\Omega} |x \cdot \xi|$. 

\begin{definition}
Let $s\ge 0$. 
For $f:\Omega\to\mathbb{R}$ define
\begin{align*}
\|f\|_{\mathscr{S}_s(\Omega)}
\coloneqq
\inf\Bigg\{
&\int_{\mathbb{R}^d} (1+\langle \xi\rangle^s) |\mathscr{F}\tilde f|(\diff\xi)\colon\\
&\tilde f|_{\Omega}=f, \mathscr{F}\tilde f \text{ is a finite signed Radon measure}
\Bigg\},
\end{align*}
where $|\mathscr{F}\tilde f|$ denotes the total variation measure of $\mathscr{F} \tilde f$.

We set $\mathscr{S}_s(\Omega)\coloneqq \{f\colon\|f\|_{\mathscr{S}_s(\Omega)}<\infty\}$.
\end{definition}

\begin{definition}
Let $s\ge 0$. For a probability measure
$\chi\in\mathcal{P}(\mathbb{R}^{2d+1})$ define
\[
f_{\chi}(x)\coloneqq\int_{\mathbb{R}^{2d+1}}
w(a\cdot x + b)_+ \diff\chi(w,a,b).
\]
For $f:\Omega\to\mathbb{R}$ define the admissible set $\mathscr{A}_f\coloneqq\{\chi\in\mathcal{P}(\mathbb{R}^{2d+1})\colon  f_{\chi}|_{\Omega}=f\},$
and the norm
\[
\|f\|_{\mathscr{B}_s(\Omega)}
\coloneqq 
\inf_{\chi\in\mathscr{A}_f}
\int |w|(\langle a\rangle+|b|)^s \diff\chi(w, a, b).
\]
We set $\mathscr{B}_s(\Omega)\coloneqq\{f\colon \|f\|_{\mathscr{B}_s(\Omega)}<\infty\}$---the \emph{Barron} space.
\end{definition}

\begin{remark}[Vector-valued extension]
For $f=(f^1,\dots,f^m):\Omega\to\mathbb{R}^m$ we use the componentwise norms
\[
\|f\|_{\mathscr{B}_s(\Omega)}\coloneqq\sum_{j=1}^m \|f^j\|_{\mathscr{B}_s(\Omega)},
\qquad
\|f\|_{\mathscr{S}_s(\Omega)}\coloneqq \sum_{j=1}^m \|f^j\|_{\mathscr{S}_s(\Omega)}.
\]
All embeddings below apply componentwise.
\end{remark}

\subsection{The proof}

\begin{proof}[Proof of \Cref{thm:sobolev_to_switches_TV}]

The proof is split in several steps.

\subsubsection*{Step 1.}

For $s>d/2+1$, there exists $u\in L^2([0,1];H^s(\mathbb{R}^d;\mathbb{R}^d))$ such that
$\phi_u^1=\phi$ and
\[
\int_0^1 \|u(t)\|_{H^s}^2 \diff t \le \mathsf{A}_s(\phi)+1,
\]
where we add $1$ only to avoid discussing attainment of the infimum. Sobolev embedding gives $u(t)\in L^\infty \cap C^{1}(\R^d)$ for a.e.\ $t$,
hence $u(t,\cdot)$ is globally Lipschitz and the flow $X(t,x)$ of
$\dot x=u(t,x)$ exists globally and is a $C^1$ diffeomorphism for each $t$.
Moreover, there is a constant $C_{d,s}>0$ such that
\[
\|u(t)\|_{L^\infty}+\|\nabla u(t)\|_{L^\infty}\le C_{d,s}\|u(t)\|_{H^s}.
\]
In particular $t\mapsto \|u(t)\|_{L^\infty}$ is in $L^1(0,1)$, so the trajectories
starting in $\Omega$ remain in a fixed ball:
define
\[
R_*\coloneqq 1+\sup_{x\in\Omega}|x| + \int_0^1 \|u(t)\|_{L^\infty}\diff t,
\qquad
\widetilde\Omega\coloneqq B(0,R_*).
\]
Then $X(t,\Omega)\subset\widetilde\Omega$ for all $t\in[0,1]$.

We will build the approximation so that also $Y(t,\Omega)\subset\widetilde\Omega$
for all $t$, hence all evaluations of $u$ and $\nabla_x\cdot u$ along both flows occur
inside $\widetilde\Omega$.

\subsubsection*{Step 2.}

We claim that, for $s> d/2+2$ and bounded $\Omega$, there is $C_{d,s,\Omega}$ such that
for every $f\in H^s(\mathbb{R}^d)$,
\begin{equation} \label{eq:sobolev_to_F2}
\|f\|_{\mathscr{S}_2(\Omega)}
\le C_{d,s,\Omega}\|f\|_{H^s(\mathbb{R}^d)}.
\end{equation}
Indeed, fix $\alpha\coloneqq s-2>d/2$. For a (suitable) extension $\tilde f$ of $f$ with
$\mathscr{F} \tilde f\in L^2$ we estimate by Cauchy--Schwarz
\begin{align*}
\int_{\mathbb{R}^d} (1+\langle \xi\rangle)^2 |\mathscr{F}\tilde f(\xi)|\diff \xi
\le
&\left(\int_{\mathbb{R}^d} (1+\langle \xi\rangle)^{-2\alpha}\diff \xi\right)^{\frac12}\\
&\left(\int_{\mathbb{R}^d} (1+\langle\xi\rangle)^{2(2+\alpha)} |\mathscr{F} \tilde f(\xi)|^2\diff\xi\right)^{\frac12},
\end{align*}
and the first integral is finite since $\alpha> d/2$, while the second term is controlled by
$\|f\|_{H^{2+\alpha}}=\|f\|_{H^s}$. Taking the infimum over admissible extensions yields
\eqref{eq:sobolev_to_F2}.

Next, by the embedding inequalities between $\mathscr{S}_s$ and $\mathscr{B}_s$
(see \cite[Theorem~1.1]{wu2023embedding}),
\begin{equation}\label{eq:F2_to_B1}
\|f\|_{\mathscr{B}_1(\Omega)} \le C_{d,\Omega}\|f\|_{\mathscr{S}_2(\Omega)}.
\end{equation}
Applying \eqref{eq:sobolev_to_F2}--\eqref{eq:F2_to_B1} componentwise to $u(t,\cdot)$ gives
\[
\|u(t,\cdot)\|_{\mathscr{B}_1(\Omega)}
\le C_{d,s,\Omega}\|u(t,\cdot)\|_{H^s(\mathbb{R}^d)}\qquad\text{for a.e.\ }t,
\]
and therefore 
\[
\mathsf{E}_B(u)^2
\coloneqq
\int_0^1 \|u(t,\cdot)\|_{\mathscr{B}_1(\Omega)}^2\diff t
\]
satisfies
\[
\mathsf{E}_B(u)^2
\le C_{d,s,\Omega}\int_0^1 \|u(t)\|_{H^s}^2 \diff  t
\le C_{d,s,\Omega}(\mathsf{A}_s(\phi)+1).
\]
Set $R\coloneqq 1+\sup_{x\in\widetilde\Omega}|x|$. For $\theta=(w,a,b)\in\R^{2d+1}$ define
\begin{equation}\label{eq: cost.Maurey}
c(\theta)\coloneqq |w|(R|a|+|b|).
\end{equation}
Define
\[
g_\theta(x)\coloneqq
\begin{cases}
\dfrac{w(a\cdot x+b)_+}{c(\theta)} & c(\theta)>0,\\[0.05em]
0 & c(\theta)=0,
\end{cases}
\quad
h_\theta(x)\coloneqq
\begin{cases}
\dfrac{(w\cdot a)\mathbf 1_{\{a\cdot x+b>0\}}}{c(\theta)} & c(\theta)>0,\\[0.05em]
0 & c(\theta)=0.
\end{cases}
\]
Then for all $\theta$:
\begin{equation}\label{eq:bounds_atoms}
 \|g_\theta\|_{L^\infty(\widetilde\Omega)}\le 1,
\qquad
\|g_\theta\|_{C^{0,1}(\widetilde\Omega)}\le 1,
\qquad
\|h_\theta\|_{L^\infty(\widetilde\Omega)}\le 1.
\end{equation}
(The Lipschitz bound for $g_\theta$ uses that $x\mapsto(a\cdot x+b)_+$ is
$|a|$-Lipschitz and $|a\cdot x+b|\le R|a|+|b|$ on $\widetilde\Omega$.)

By definition of the Barron norm and a measurable selection argument,
for a.e.\ $t$ there exists a probability measure $\chi_t$ on $\R^{2d+1}$ such that
\begin{equation}\label{eq:rep_u_full_dom}
u(t,x)=\int_{\R^{2d+1}} w(a\cdot x+b)_+\diff \chi_t(\theta)
=\int_{\R^{2d+1}} c(\theta)g_\theta(x)\diff \chi_t(\theta)
,
\end{equation}
for all $ x\in\widetilde\Omega$ and
\[
r(t)\coloneqq \int_{\R^{2d+1}} c(\theta)\diff \chi_t(\theta)
\le 2\|u(t,\cdot)\|_{\mathscr{B}_1(\widetilde\Omega)}.
\]
Hence $r\in L^2(0,1)$ and
\begin{equation}\label{eq: r.bound.action}
\|r\|_{L^2(0,1)}^2 \lesssim \int_0^1 \|u(t)\|_{H^s}^2\diff t
\lesssim \mathsf{A}_s(\phi)+1.
\end{equation}
In particular, one has 
\begin{equation}\label{eq:u-sup-lip-by-r}
\|u(t,\cdot)\|_{L^\infty(\widetilde\Omega)}\le r(t),
\qquad
\|u(t,\cdot)\|_{C^{0,1}(\widetilde\Omega)}\le \frac{r(t)}{R}.
\end{equation}
Moreover, differentiating under the integral in \eqref{eq:rep_u_full_dom}
is justified by dominated convergence (since $|(w\cdot a)\mathbf 1_{\{a\cdot x+b>0\}}|
\le |w||a|\le c(\theta)$ and $\int c\diff \chi_t=r(t)<\infty$),
and we obtain for a.e.\ $x\in\widetilde\Omega$:
\begin{align}\label{eq:rep_div_full_dom}
\nabla_x\cdot u(t,x)
&=
\int_{\R^{2d+1}} (w\cdot a)\,\mathbf 1_{\{a\cdot x+b>0\}}\diff \chi_t(\theta)
\nonumber\\
&=
\int_{\R^{2d+1}} c(\theta)\,h_\theta(x)\diff \chi_t(\theta).
\end{align}
Finally, since $s>d/2+2$, we have $\nabla_x\cdot u(t)\in L^\infty\cap C^{0,1}(\R^d)$ for a.e.\ $t$ and
there exists $C_{d,s}>0$ such that
\begin{equation}\label{eq:div_Lip_def}
\|\nabla_x\cdot u(t)\|_{C^{0,1}(\widetilde\Omega)}\le C_{d,s}\|u(t)\|_{H^s}.
\end{equation}
Define
\[
\ell(t)\coloneqq \|\nabla_x\cdot u(t)\|_{C^{0,1}(\widetilde\Omega)}\in L^2(0,1).
\]

\subsubsection*{Step 3.}
Fix $N\in\N$, set $\Delta\coloneqq 1/N$ and $I_k=[(k-1)\Delta,k\Delta)$.
Define the finite measures
\[
M_k(A)\coloneqq \int_{I_k}\int_A c(\theta)\diff \chi_t(\theta)\diff t,
\qquad
r_k\coloneqq M_k(\R^{2d+1})=\int_{I_k} r(t)\diff t.
\]
If $r_k>0$ set $\nu_k\coloneqq M_k/r_k$; otherwise fix an arbitrary $\nu_k$.
Sample $\theta_k\sim\nu_k$ independently for $k=1,\dots,N$ and define the
piecewise-constant vector field
\[
u_N(t,x)\coloneqq N\,r_k\,g_{\theta_k}(x)\qquad(t\in I_k).
\]
Since $g_{\theta_k}(x)=w_k(a_k\cdot x+b_k)_+/c(\theta_k)$, on $I_k$ we can rewrite
\[
u_N(t,x)=w'_k\,(a_k\cdot x+b_k)_+,
\qquad
w'_k\coloneqq \frac{N r_k}{c(\theta_k)}\,w_k.
\]
Thus $u_N$ is implemented by a single neuron with frozen parameters on each $I_k$,
hence the control has exactly $N$ switches.

Let $X(t,x)$ be the flow of $u$ and $Y(t,x)$ the flow of $u_N$, with
$X(0,x)=Y(0,x)=x$. Set $t_k\coloneqq k\Delta$ and define $X_k\coloneqq X(t_k,\cdot)$, $Y_k\coloneqq Y(t_k,\cdot)$.

\subsubsection*{Step 3.1.}
Define
\[
R_k^X(x)\coloneqq\int_{I_k}\left(u(t,X(t,x))-u(t,X_{k-1}(x))\right)\diff t,
\]
\[
R_k^Y(x)\coloneqq \int_{I_k}\left(u_N(t,Y(t,x))-u_N(t,Y_{k-1}(x))\right)\diff t.
\]
Then
\[
X_k=X_{k-1}+\int_{I_k}u(t,X_{k-1})\diff t+R_k^X,
\qquad
Y_k=Y_{k-1}+\int_{I_k}u_N(t,Y_{k-1})\diff t+R_k^Y.
\]
Since $u_N=N r_k g_{\theta_k}$ on $I_k$ and $|I_k|=\Delta$, we have
$\int_{I_k}u_N(t,\cdot)\diff t=r_k g_{\theta_k}(\cdot)$, so
\[
Y_k=Y_{k-1}+r_k g_{\theta_k}(Y_{k-1})+R_k^Y.
\]
Set $e_k\coloneqq X_k-Y_k$ and define the fluctuation
\[
\eta_k(x)\coloneqq \int_{I_k}u(t,Y_{k-1}(x))\diff t-r_k g_{\theta_k}(Y_{k-1}(x)).
\]
Subtracting and adding $\int_{I_k}u(t,Y_{k-1})\diff t$ yields
\begin{equation}\label{eq:e_rec}
e_k
=
e_{k-1}
+\underbrace{\int_{I_k}\left(u(t,X_{k-1})-u(t,Y_{k-1})\right)\diff t}_{=:A_k}
+\eta_k
+\underbrace{(R_k^X-R_k^Y)}_{=:R_k}.
\end{equation}
For \(k\ge 1\), let \(\mathcal F_{k-1}\) denote the collection of events that are
completely determined by the first \(k-1\) samples \((\theta_1,\dots,\theta_{k-1})\), namely
\[
\mathcal F_{k-1}
\coloneqq
\left\{\,(\theta_1,\dots,\theta_{k-1})^{-1}(B)\ :\ B\in\mathscr B\left((\R^{2d+1})^{k-1}\right)\,\right\}.
\]
(Equivalently, \(\mathcal F_{k-1}\) is the smallest \(\sigma\)-algebra for which each
\(\theta_1,\dots,\theta_{k-1}\) is measurable.) Using \eqref{eq:u-sup-lip-by-r} we have
\[
\|A_k\|_{L^2(\Omega)}
\le \int_{I_k}\|u(t,\cdot)\|_{C^{0,1}(\widetilde\Omega)}\diff t\ \|e_{k-1}\|_{L^2(\Omega)}
\le \frac{r_k}{R}\|e_{k-1}\|_{L^2(\Omega)}.
\]
Also, for $t\in I_k$ and $x\in\Omega$,
\[
|X(t,x)-X_{k-1}(x)|\le \int_{I_k} \|u(s,\cdot)\|_{L^\infty(\widetilde\Omega)}\diff s \le r_k,
\]
hence
\[
|R_k^X(x)| \le \int_{I_k}\|u(t,\cdot)\|_{C^{0,1}(\widetilde\Omega)}|X(t,x)-X_{k-1}(x)|\diff t
\le \int_{I_k}\frac{r(t)}{R}\diff t\ r_k
\le \frac{r_k^2}{R}.
\]
Similarly, using \eqref{eq:bounds_atoms}, $|g_{\theta_k}|\le 1$ on $\widetilde\Omega$ and $Y(t,x)\in\widetilde\Omega$, we have
$|u_N(t,Y(t,x))|\le N r_k$ on $I_k$, hence $|Y(t,x)-Y_{k-1}(x)|\le r_k$ and
\[
|R_k^Y(x)|
\le \int_{I_k}\|u_N(t,\cdot)\|_{C^{0,1}(\widetilde\Omega)}|Y(t,x)-Y_{k-1}(x)|\diff t
\le \int_{I_k}\frac{N r_k}{R}\diff t r_k
= \frac{r_k^2}{R}.
\]
Therefore,
\begin{equation}\label{eq:step3-Rk-bound-rigorous}
\|R_k\|_{L^2(\Omega)}\le \frac{2|\Omega|^{1/2}}{R}r_k^2.
\end{equation}
Conditionally on $\mathcal F_{k-1}$, $Y_{k-1}$ is deterministic and $\theta_k\sim\nu_k$ is independent.
By construction,
\[
\mathbb{E}[r_k g_{\theta_k}(z)\mid\mathcal F_{k-1}]
=r_k\int g_\theta(z)\diff\nu_k(\theta)=\int g_\theta(z)\diff M_k(\theta)=\int_{I_k}u(t,z)\diff t
\]
for every $z\in\widetilde\Omega$.
Applying this with $z=Y_{k-1}(x)\in\widetilde\Omega$ yields $\mathbb{E}[\eta_k(x)\mid\mathcal F_{k-1}]=0$.
Moreover $|g_{\theta_k}|\le 1$ on $\widetilde\Omega$ implies $|r_k g_{\theta_k}(Y_{k-1}(x))|\le r_k$, hence
\begin{equation}\label{eq:step3-eta-var-rigorous}
\mathbb{E}\left[\|\eta_k\|_{L^2(\Omega)}^2\mid\mathcal F_{k-1}\right]\le |\Omega| r_k^2.
\end{equation}
Taking $L^2$-norm squared in \eqref{eq:e_rec} and conditioning on $\mathcal F_{k-1}$, we expand
in $L^2(\Omega)$:
\begin{align*}
\mathbb{E}\Big[\|e_k\|_{L^2}^2\mid&\mathcal F_{k-1}\Big]
=\|e_{k-1}+A_k+R_k\|_{L^2}^2
\\
&+2\left\langle e_{k-1}+A_k+R_k, \mathbb{E}[\eta_k\mid\mathcal F_{k-1}]\right\rangle +\mathbb{E}\left[\|\eta_k\|_{L^2}^2\mid\mathcal F_{k-1}\right].
\end{align*}
Since $\mathbb{E}[\eta_k\mid\mathcal F_{k-1}]=0$, the cross term vanishes and \eqref{eq:step3-eta-var-rigorous} gives
\begin{equation}\label{eq:step3-cond-energy-rigorous}
\mathbb{E}\left[\|e_k\|_{L^2}^2\mid\mathcal F_{k-1}\right]
\le \|e_{k-1}+A_k+R_k\|_{L^2}^2 + |\Omega| r_k^2.
\end{equation}
We bound $\|e_{k-1}+A_k+R_k\|_{L^2}^2$ as follows.
First, using the bound on $A_k$,
\[
\|e_{k-1}+A_k\|_{L^2}\le \left(1+\frac{r_k}{R}\right)\|e_{k-1}\|_{L^2}.
\]
Second, for any $\varepsilon>0$ we use  Young's inequality
\[
\|u+v\|_{L^2}^2\le (1+\varepsilon)\|u\|_{L^2}^2+\left(1+\frac1\varepsilon\right)\|v\|_{L^2}^2.
\]
Apply this with $u=e_{k-1}+A_k$, $v=R_k$, and $\varepsilon=r_k/R$.
Using \eqref{eq:step3-Rk-bound-rigorous} yields
\[
\|e_{k-1}+A_k+R_k\|_{L^2}^2
\le \left(1+\frac{r_k}{R}\right)\left(1+\frac{r_k}{R}\right)^2\|e_{k-1}\|_{L^2}^2
+ C_\Omega\left(\frac{r_k^3}{R^2}+\frac{r_k^4}{R^3}\right),
\]
hence
\[
\|e_{k-1}+A_k+R_k\|_{L^2}^2
\le \left(1+\frac{r_k}{R}\right)^3\|e_{k-1}\|_{L^2}^2
+ C_\Omega\left(\frac{r_k^3}{R^2}+\frac{r_k^4}{R^3}\right).
\]
Since $(1+x)^3\le e^{3x}$, we conclude from \eqref{eq:step3-cond-energy-rigorous} (taking expectations) that
\begin{equation}\label{eq:step3-energy-rec-rigorous}
\mathbb{E}\|e_k\|_{L^2(\Omega)}^2
\le e^{C r_k/R} \mathbb{E}\|e_{k-1}\|_{L^2(\Omega)}^2
+ C_\Omega\left(r_k^2+\frac{r_k^3}{R}+\frac{r_k^4}{R^2}\right).
\end{equation}
Iterating \eqref{eq:step3-energy-rec-rigorous} yields
\[
\mathbb{E}\|e_N\|_{L^2(\Omega)}^2
\le
C_\Omega \exp\left(C\sum_{k=1}^N \frac{r_k}{R}\right)\sum_{k=1}^N \left(r_k^2+\frac{r_k^3}{R}+\frac{r_k^4}{R^2}\right).
\]
Since $\sum_{k=1}^N r_k=\|r\|_{L^1(0,1)}$ and $R\ge 1$, we may bound
$$e^{\frac{C}{R}\sum r_k}\le e^{C\|r\|_{L^1}}$$ and absorb the $R$-factors.
Moreover, by Cauchy--Schwarz,
\begin{equation}\label{eq: int.cw}
    |r_k|=\left|\int_{I_k}r(t)\diff t\right|\leq \left(\int_{I_k}r(t)^2\diff t\right)^{\frac12}|I_k|^{\frac{1}{2}}=\frac{\left(\int_{I_k}r(t)^2\diff t\right)^2}{N^{\frac12}}
\end{equation}
hence
\[
r_k^2=\frac1N\int_{I_k} r(t)^2\diff t,
\]
therefore
$$\sum_{k=1}^N r_k^2\le \frac{\|r\|_{L^2(0,1)}^2}{N}$$
and using \eqref{eq: int.cw} for the higher-order terms we can bound
$$\sum_{k=1}^N|r_k|^{3}\leq \sum_{k=1}^N \left(\int_{I_k} r(t)^2\diff t\right)^{\frac32}N^{-\frac32}\leq N^{-\frac32}\left(\sum_{k=1}^N\int_{I_k}r(t)^2\diff t\right)^{\frac32}\leq \frac{\|r\|_{L^2(0,1)}^3}{N^{\frac32}}$$
and
$$\sum_{k=1}^N|r_k|^{4}\leq \sum_{k=1}^N \left(\int_{I_k} r(t)^2\diff t\right)^{2}N^{-2}\leq N^{-2}\left(\sum_{k=1}^N\int_{I_k}r(t)^2\diff t\right)^{2}\leq \frac{\|r\|_{L^2(0,1)}^4}{N^2}.$$
Since the third and fourth order powers of $r_k$ are of $o(N^{-1})$, for high enough $N$, using Cauchy-Schwarz we can bound we can bound
\begin{equation}\label{eq:Ek_flow_rate}
    \mathbb{E}\|e_N\|_{L^2(\Omega)}
\le C_\Omega \exp(C\|r\|_{L^1(0,1)})\frac{\|r\|_{L^2(0,1)}}{\sqrt N}.
\end{equation}
Consequently there exists a deterministic realization $(\theta_1,\dots,\theta_N)$ such that
\begin{equation*}
    \|Y(1,\cdot)-X(1,\cdot)\|_{L^2(\Omega)}
\le C_\Omega \exp(C\|r\|_{L^1})\frac{\|r\|_{L^2}}{\sqrt N}.
\end{equation*}
Finally, since \eqref{eq: r.bound.action}, we end up with
\begin{equation*}
    \|Y(1,\cdot)-X(1,\cdot)\|_{L^2(\Omega)}
\le C_\Omega \left(C\sqrt{\mathsf{A}(\phi)+1}\right)\frac{\sqrt{\mathsf{A}(\phi)+1}}{\sqrt N}.
\end{equation*}

\subsubsection*{Step 3.2.}
Define
\[
q_X(t,x)\coloneqq \log\det\nabla_x X(t,x),
\qquad
q_Y(t,x)\coloneqq \log\det\nabla_x Y(t,x).
\]
Set
\[
A_X(t,x)\coloneqq \nabla_x X(t,x)\in\mathbb{R}^{d\times d}.
\]
Differentiate the ODE $\partial_t X(t,x)=u(t,X(t,x))$ with respect to $x$.
For each $j\in\{1,\dots,d\}$, the vector $\partial_{x_j}X(t,x)$ satisfies
\[
\begin{cases}
    \partial_t\left(\partial_{x_j}X(t,x)\right)
=\nabla u\left(t,X(t,x)\right)\,\partial_{x_j}X(t,x)\\
\partial_{x_j}X(0,x)=e_j
\end{cases}
\]
where $\nabla u(t,\cdot)$ denotes the Jacobian matrix of $u(t,\cdot)$.
Collecting these $d$ equations column-wise yields the matrix ODE
\begin{equation}\label{eq:variational-X}
\begin{cases}
\partial_t A_X(t,x)
=\nabla u\left(t,X(t,x)\right)\,A_X(t,x)\\
A_X(0,x)=I_d.
\end{cases}
\end{equation}
Analogously, for
\(
A_Y(t,x)\coloneqq \nabla_x Y(t,x),
\)
we have
\begin{equation}\label{eq:variational-Y}
\begin{cases}
\partial_t A_Y(t,x)
=\nabla u_N\left(t,Y(t,x)\right)\,A_Y(t,x)\\
A_Y(0,x)=I_d.
\end{cases}
\end{equation}
As $A_X(t,x)$ is invertible, Jacobi's formula gives
\[
\partial_t \det A_X(t,x)
=\det A_X(t,x)\,\mathrm{tr}\left(A_X(t,x)^{-1}\,\partial_t A_X(t,x)\right).
\]
Therefore,
\begin{align*}
\partial_t q_X(t,x)
&=\partial_t \log\det A_X(t,x)
=\mathrm{tr}\left(A_X(t,x)^{-1}\,\partial_t A_X(t,x)\right)\\
&=\mathrm{tr}\left(A_X(t,x)^{-1}\,\nabla u\left(t,X(t,x)\right)\,A_X(t,x)\right)
=\mathrm{tr}\left(\nabla u\left(t,X(t,x)\right)\right)\\
&=\nabla\cdot u\left(t,X(t,x)\right),
\end{align*}
where we used \eqref{eq:variational-X} and invariance of the trace under similarity.
The initial condition is
\[
q_X(0,x)=\log\det\nabla_x X(0,x)=\log\det I_d=0.
\]
In particular, for all $t$,
\[
q_X(t,x)=\int_0^t \nabla\cdot u\left(s,X(s,x)\right)\diff s.
\]
Similarly, using \eqref{eq:variational-Y},
\[
\begin{cases}
\partial_t q_Y(t,x)=\nabla\cdot u_N\left(t,Y(t,x)\right)\\
q_Y(0,x)=0
\end{cases}
\]
and
\[\quad
q_Y(t,x)=\int_0^t \nabla\cdot u_N\left(s,Y(s,x)\right)\diff s.
\]
On each $I_k$, the field $u_N(t,x)=w'_k(a_k\cdot x+b_k)_+$ is affine on each side
of the hyperplane $\{a_k\cdot x+b_k=0\}$ and globally Lipschitz.
Moreover, along any trajectory $t\mapsto Y(t,x)$ on $I_k$ the sign of
$a_k\cdot Y(t,x)+b_k$ is constant:
indeed $z(t)\coloneqq a_k\cdot Y(t,x)+b_k$ satisfies $z'(t)=(a_k\cdot w'_k)\,z(t)_+$,
so if $z(t_0)\le0$ then $z'(t)=0$ and $z(t)\le0$ for all $t\ge t_0$, while if
$z(t_0)>0$ then $z(t)=z(t_0)e^{(a_k\cdot w'_k)(t-t_0)}>0$.
Therefore 
$$\nabla_x\cdot u_N(t,Y(t,x))=(w'_k\cdot a_k)\mathbf 1_{\{a_k\cdot Y_{k-1}(x)+b_k>0\}}\text{ is constant in $t\in I_k$ }$$
for a.e.\ $x$, and using $w'_k=(Nr_k/c(\theta_k))w_k$ we get
\[
\nabla_x\cdot u_N(t,Y(t,x))=N r_k\,h_{\theta_k}(Y_{k-1}(x))
\qquad (t\in I_k),
\]
hence
\begin{align}\label{eq:qY_increment_exact}
q_Y(t_k,x)&=q_Y(t_{k-1},x)+\int_{I_k}\nabla_x\cdot u_N(t,Y(t,x))\diff t\nonumber\\
&=q_Y(t_{k-1},x)+r_kh_{\theta_k}(Y_{k-1}(x))
\end{align}
for a.e.\ $x$\footnote{Observe that \eqref{eq:qY_increment_exact} does not have any remainder term due to the fact that the divergence is constant along the trajectory. This is specific for the ReLU activation or variants.}.

For $q_X$ we write
\[
q_X(t_k,x)=q_X(t_{k-1},x)+\int_{I_k}\nabla_x\cdot u(t,X_{k-1}(x))\diff t + S_k^X(x),
\]
where the remainder
\[
S_k^X(x)\coloneqq \int_{I_k}\left(\nabla_x\cdot u(t,X(t,x))-\nabla_x\cdot u(t,X_{k-1}(x))\right)\diff t.
\]
Using $\|\nabla_x\cdot u(t)\|_{C^{0,1}(\widetilde\Omega)}=\ell(t)$ and
$|X(t,x)-X_{k-1}(x)|\le r_k$ gives
\begin{equation}\label{eq:SX_bound}
|S_k^X(x)|\le r_k\int_{I_k}\ell(t)\diff t=:r_k\,\ell_k,
\qquad
\|S_k^X\|_{L^2(\Omega)}\le |\Omega|^{1/2}r_k\ell_k.
\end{equation}
Define the log-Jacobian error $$\delta_k\coloneqq q_X(t_k,\cdot)-q_Y(t_k,\cdot)$$
and the fluctuation
\[
\zeta_k(x)\coloneqq \int_{I_k}\nabla_x\cdot u(t,Y_{k-1}(x))\diff t - r_k\,h_{\theta_k}(Y_{k-1}(x)).
\]
Subtracting the updates for $q_X$ and $q_Y$ and adding/subtracting
$\int_{I_k}\nabla_x\cdot u(t,Y_{k-1})\diff t$ yields
\begin{equation}\label{eq:delta_rec}
\delta_k
=
\delta_{k-1}
+\underbrace{\int_{I_k}\left(\nabla_x\cdot u(t,X_{k-1})-\nabla_x\cdot u(t,Y_{k-1})\right)\diff t}_{=:B_k}
+\zeta_k
+S_k^X.
\end{equation}
Since $\nabla_x\cdot u(t,\cdot)$ is $\ell(t)$-Lipschitz on $\widetilde\Omega$,
\[
\|B_k\|_{L^2(\Omega)}\le \ell_k\,\|e_{k-1}\|_{L^2(\Omega)}.
\]
As before, conditionally on $\mathcal F_{k-1}$ we have
$\mathbb E[\zeta_k(\cdot)\mid\mathcal F_{k-1}]=0$ because by \eqref{eq:rep_div_full_dom},
for every $z\in\widetilde\Omega$,
\begin{align*}
    \mathbb E[r_k h_{\theta_k}(z)\mid\mathcal F_{k-1}]
&=r_k\int h_\theta(z)\diff \nu_k(\theta)\\
&=\int h_\theta(z)\diff M_k(\theta)
=\int_{I_k}\nabla_x\cdot u(t,z)\diff t.
\end{align*}
Moreover $|h_{\theta_k}|\le1$ implies
$\mathbb E[\|\zeta_k\|_{L^2(\Omega)}^2\mid\mathcal F_{k-1}]\le |\Omega|\,r_k^2$.

A conditional energy estimate in $L^2(\Omega)$ applied to \eqref{eq:delta_rec},
using \eqref{eq:SX_bound} and the bound on $B_k$, yields a recursion of the form
\[
\mathbb E\|\delta_k\|_{L^2(\Omega)}^2
\le
e^{C\ell_k}\,\mathbb E\|\delta_{k-1}\|_{L^2(\Omega)}^2
+C\,\ell_k\,\mathbb E\|e_{k-1}\|_{L^2(\Omega)}^2
+C_\Omega\,(r_k^2+r_k^2\ell_k^2).
\]
Iterating, and using \eqref{eq:Ek_flow_rate} to control $\mathbb E\|e_{k-1}\|^2$,
together with the bounds
$\sum_k r_k^2\lesssim \|r\|_{L^2}^2/N$ and $\sum_k \ell_k^2\lesssim \|\ell\|_{L^2}^2/N$,
gives
\begin{equation*}
\mathbb E\|\delta_N\|_{L^2(\Omega)}
\le
C_\Omega\,\exp\left(C(\|r\|_{L^1(0,1)}+\|\ell\|_{L^1})\right)\,
\frac{\|r\|_{L^2(0,1)}+\|\ell\|_{L^2}}{\sqrt N}.
\end{equation*}
By \eqref{eq:div_Lip_def},
$\|\ell\|_{L^2}\lesssim \|u\|_{L^2H^s}\lesssim \sqrt{\mathsf{A}_s(\phi)+1}$ and similarly
$\|r\|_{L^2}\lesssim \sqrt{\mathsf{A}_s(\phi)+1}$, while $\|r\|_{L^1},\|\ell\|_{L^1}$
are controlled by the same quantity by Cauchy--Schwarz.

Therefore, there exists at least one deterministic realization
$(\theta_1,\dots,\theta_N)$ such that simultaneously
\begin{equation}\label{eq:det_realization_both}
\|e_N\|_{L^2(\Omega)}+\|\delta_N\|_{L^2(\Omega)}
\le
C\,\exp\left(C\sqrt{\mathsf{A}_s(\phi)}\right)\,\frac{\sqrt{\mathsf{A}_s(\phi)}}{\sqrt N}.
\end{equation}
Fix such a realization and let $\phi_\theta^1\coloneqq Y(1,\cdot)$ be the corresponding flow map.

\subsubsection*{Step 4.}
Set $\omega\coloneqq \supp\rho_0\subset\Omega$ and let $\psi\coloneqq \phi_\theta^1$.
Define $\eta\coloneqq \phi^{-1}\circ\psi$. First, since $\phi^{-1}$ is Lipschitz on $\psi(\omega)\cup \phi(\omega)$ and
$\psi$ is close to $\phi$ on $\omega$, we have
\begin{align*}
\|\eta-\Id\|_{L^1(\omega)}
&=
\|\phi^{-1}\circ\psi-\phi^{-1}\circ\phi\|_{L^1(\omega)}\\
&\le \|\phi^{-1}\|_{C^{0,1}}\,\|\psi-\phi\|_{L^1(\omega)}
\le C\,\|e_N\|_{L^2(\Omega)}.
\end{align*}
Next, observe the chain rule identity
\[
\log\det\nabla\eta(x)
=
\log\det\nabla\psi(x) - \log\det\nabla\phi(\eta(x))
=
q_Y(1,x)-q_X(1,\eta(x)).
\]
Thus
\begin{align*}
    |\log\det\nabla\eta(x)|
&\le |q_Y(1,x)-q_X(1,x)|
+|q_X(1,x)-q_X(1,\eta(x))|\\
&\le |\delta_N(x)|+\|q_X(1,\cdot)\|_{C^{0,1}}\,|\eta(x)-x|.
\end{align*}
Integrating over $\omega$ gives
\begin{align*}
   \|\log\det\nabla\eta\|_{L^1(\omega)}
&\le
\|\delta_N\|_{L^1(\omega)} + \|q_X(1,\cdot)\|_{C^{0,1}}\,\|\eta-\Id\|_{L^1(\omega)}\\
&\le C\left(\|\delta_N\|_{L^2(\Omega)}+\|e_N\|_{L^2(\Omega)}\right). 
\end{align*}
Moreover, $q_X(1,\cdot)$ is bounded on $\widetilde\Omega$ by
$\int_0^1\|\nabla_x\cdot u(t)\|_{L^\infty}\diff t<\infty$, and $q_Y(1,\cdot)$ is bounded by
$\sum_k r_k=\|r\|_{L^1}$, so $\|\log\det\nabla\eta\|_{L^\infty(\omega)}$ is bounded by
a constant depending only on $\phi$ and the chosen $u$ (hence on $\mathsf{A}_s(\phi)$).

Therefore we can apply \Cref{lem:TV_stability_diffeo} with $\rho=\rho_0$ and obtain
\begin{align*}
    |\phi_\#\mu_0-\psi_\#\mu_0|_{\mathsf{TV}}
&\le
C_{\rho_0,\phi}\left(\|\eta-\Id\|_{L^1(\omega)}+\|\log\det\nabla\eta\|_{L^1(\omega)}\right)\\
&\le
C\left(\|e_N\|_{L^2(\Omega)}+\|\delta_N\|_{L^2(\Omega)}\right).
\end{align*}
Combining this with \eqref{eq:det_realization_both} yields
\[
|\phi_\#\mu_0-\psi_\#\mu_0|_{\mathsf{TV}}
\le
C\,\exp\left(C\sqrt{\mathsf{A}_s(\phi)}\right)\,\frac{\sqrt{\mathsf{A}_s(\phi)}}{\sqrt N}.\qedhere
\]
\end{proof}

\subsection{Discussion} \label{sec: discussion}

The Maurey empirical method \cite{pisier1981remarques} is an infinite-dimensional avatar of the Carathéodory theorem and says that if $h$ belongs to the closed convex hull of a bounded set
$S$ in a Hilbert space $H$, then $h$ can be approximated by an average of
$m$ points in $S$ with error $\lesssim m^{-1/2}$. In approximation theory for Barron functions,
this is exactly the Monte--Carlo proof of the width--$m$ rate
$\|f-f_m\|\lesssim m^{-1/2}$, where $f_m$ is a two-layer (width $m$) perceptron obtained by
sampling $m$ neurons from a representing Barron measure. The construction in Step~3 uses the same  mechanism, but the $m$ samples
are placed along time rather than along width: we keep width $=1$, and we sample $N$ neurons across $N$ time-intervals, producing a piecewise-constant
control with $N$ switches and an $N^{-1/2}$ error---hence $N\sim \varepsilon^{-2}$.

A surprising feature of the argument leading to the  $\mathsf{TV}$--control in the spirit of
\Cref{thm:sobolev_to_switches} is that we do not impose an explicit polar--type structure on the target map, as we do in
\Cref{thm: gen.thm} nor we take explicit care on controlling the compressible part of the map in its construction.  
The structure is already encoded in \eqref{eq: cost.Maurey} and therefore in the Maurey method itself. 
Crucially, the ReLU makes the log--Jacobian increment on each time interval $I_k$ explicit. This is not the case for arbitrary activation functions. 
A naive Maurey-in-time scheme may still approximate the map, but it can fail to provide any uniform control on the log--Jacobian unless one redesigns the cost \eqref{eq: cost.Maurey} and keeps track of additional remainder terms. 

We provide some detail on this discussion. Let us go back to \eqref{eq:qY_increment_exact} and consider a within--interval remainder:
\begin{align*}
q_Y(t_k,x)-q_Y(t_{k-1},x)&=\int_{I_k}\nabla_x\cdot u_N(t,Y(t,x))\diff t=r_kh_{\theta_k}(Y_{k-1}(x)).
\end{align*}
As mentioned in the proof, this equality is only true for the ReLU. For a general Lipschitz activation,
\[
q_Y(t_k,x)-q_Y(t_{k-1},x)
= r_kh_{\theta_k}\left(Y(t_{k-1},x)\right)+\varepsilon_k(x),
\]
where we obtain an additional remainder
\[
\varepsilon_k(x)\coloneqq \int_{I_k}\left(h_{\theta_k}(Y(t,x))-h_{\theta_k}(Y(t_{k-1},x))\right)\diff t.
\]
Controlling $\varepsilon_k$ requires a bound on the spatial Lipschitz constant of $h_{\theta}$ (or of $\nabla\cdot v_\theta$),
because typically $|Y(t,x)-Y(t_{k-1},x)|\lesssim r_k$ on $I_k$.
For a smooth activation $\sigma$ the chain rule gives
\[
\|\nabla_x(\nabla\cdot v_\theta)\|_{L^\infty(\widetilde\Omega)}
\lesssim |w||a|^2\|\sigma''\|_{L^\infty}.
\]
Thus, even though $\nabla\cdot v_\theta\in L^\infty$ for each fixed $\theta$,   quantitative control of the extra remainder
typically forces integrability/moment control of $|a|^2$ under $\chi_t$, i.e. we would need
\begin{equation*}
    \int_{\mathbb{R}^{2d+
    1}} |w|(|a|^2+|b|)\diff \chi_t(w,a,b)<+\infty
\end{equation*}
to control the increments, which leads to more moments controlled than what is required in Maurey's method for approximating the flow map \eqref{eq: cost.Maurey}-\eqref{eq:rep_u_full_dom}.
This is the precise sense in which ReLU is special: it removes $\varepsilon_k$ altogether, so one never needs to pay
for these higher--order ``$|a|^2$'' effects in the log--Jacobian recursion.

\section{Proof of \Cref{prop:KR.in.D}}\label{sec: proof.prop.kr}

\begin{proof} 
Write $x_{1:k}=(x_1,\dots,x_k)$.
For $i\in\{0,1\}$ let $\rho_i^{1:k}$ denote the $k$-dimensional marginal density on
$[0,1]^d$ and $\rho_i^{k\mid 1:k-1}$ the corresponding conditional density.
Since $\rho_i\in C^{s}([0,1]^d)$ and is strictly positive, all the marginals and conditionals
are well-defined, strictly positive, and belong to $C^{s}$.
Define the conditional distribution function
\[
F_{i,k}(x_k\mid x_{1:k-1})
\coloneqq \int_{0}^{x_k}\rho_i^{k\mid 1:k-1}(t\mid x_{1:k-1})\diff t.
\]
For each fixed $x_{1:k-1}$, $F_{i,k}(\,\cdot\,\mid x_{1:k-1})$ is $C^{s+1}$ and strictly increasing
from $0$ to $1$, hence a $C^{s+1}$ diffeomorphism of $(0,1)$ onto $(0,1)$. The KR map is defined recursively by
\begin{align*}
(\phi_{\mathrm{KR}})_1(x_1)&=F_{1,1}^{-1}(F_{0,1}(x_1)),\\
(\phi_{\mathrm{KR}})_k(x_{1:k})&=F_{1,k}(\cdot\mid (\phi_{\mathrm{KR}})_{1:k-1}(x_{1:k-1}))^{-1}\left(F_{0,k}(x_k\mid x_{1:k-1})\right).
\end{align*}
By induction in $k$ and the implicit function theorem applied to the equation
$$F_{1,k}((\phi_{\mathrm{KR}})_k(x_{1:k})\mid (\phi_{\mathrm{KR}})_{1:k-1}(x_{1:k-1}))=F_{0,k}(x_k\mid x_{1:k-1}),$$
one obtains $(\phi_{\mathrm{KR}})_k\in C^{s}$ for each $k$, hence $\phi_{\mathrm{KR}}\in C^{s}([0,1]^d;\mathbb R^d)$.
Moreover, differentiating in $x_k$ yields the standard diagonal formula
\[
\partial_{x_k}(\phi_{\mathrm{KR}})_k(x_{1:k})
=\frac{\rho_0^{k\mid 1:k-1}(x_k\mid x_{1:k-1})}{
\rho_1^{k\mid 1:k-1}((\phi_{\mathrm{KR}})_k(x_{1:k})\mid (\phi_{\mathrm{KR}})_{1:k-1}(x_{1:k-1}))}>0,
\]
so $x_k\mapsto (\phi_{\mathrm{KR}})_k(x_{1:k-1},x_k)$ is strictly increasing for each $x_{1:k-1}$.
Triangular monotonicity implies that $\phi_{\mathrm{KR}}$ is a bijection $(0,1)^d\to(0,1)^d$ with $C^{s}$ inverse,
i.e. a $C^{s}$ diffeomorphism of $(0,1)^d$.

Define a path of triangular maps $\phi_t:(0,1)^d\to(0,1)^d$ by componentwise convex interpolation
\[
\phi_{t,k}(x_{1:k})=(1-t)x_k+t(\phi_{\mathrm{KR}})_k(x_{1:k}),\qquad k=1,\dots,d.
\]
Each $\phi_t$ maps $(0,1)^d$ into itself and is a triangular diffeomorphism of $(0,1)^d$ because
\[
\partial_{x_k}(\phi_{\mathrm{KR}})_{t,k}=(1-t)+t\,\partial_{x_k}(\phi_{\mathrm{KR}})_k>0.
\]
Set
\[
u(t,\cdot)\coloneqq \partial_t \phi_t\circ \phi_t^{-1}=(\phi_{\mathrm{KR}}-\mathsf{id})\circ \phi_t^{-1}.
\]
We have $\phi_{\mathrm{KR}}-\mathsf{id}\in H^s((0,1)^d)$.
For $s>d/2+1$, composition with $C^{s}$ diffeomorphisms preserves $H^s$ regularity,
hence $u\in L^\infty([0,1];H^s((0,1)^d;\mathbb R^d))\subset L^2([0,1];H^s((0,1)^d;\mathbb R^d))$. 

By the Calderón--Stein extension theorem (see  \cite{jones1981quasiconformal} for example), there exists a bounded linear extension operator
$E:H^s((0,1)^d)\to H^s(\mathbb R^d)$.
Let $\widetilde u(t,\cdot)\coloneqq E[u(t,\cdot)]\in H^s(\mathbb R^d)$.
Then $\widetilde u\in L^2([0,1];H^s(\mathbb R^d;\mathbb R^d))$. As a consequence, we have that $\mathsf{A}(\phi_{\text{KR}})<+\infty$.

Let $\widetilde\phi^t$ denote the $\mathcal D^s(\mathbb R^d)$-valued flow generated by $\widetilde u$.
Existence of such a flow for $u\in L^1_t H^s_x$ with $s>d/2+1$ is standard.
Moreover, for $x\in(0,1)^d$ the curve $t\mapsto \phi_t(x)$ solves $\dot X=\widetilde u(t,X)$ because
$\phi_t(x)\in(0,1)^d$ and $\widetilde u=u$ on $(0,1)^d$; by uniqueness of solutions, $\widetilde\phi^t(x)=\phi_t(x)$.
In particular, $\widetilde\phi^1|_{(0,1)^d}=\phi_{\mathrm{KR}}$, so $\phi_{\mathrm{KR}}\in\mathcal D_0^s((0,1)^d)$.

The same argument applies with $\phi_{\mathrm{KR}}$ replaced by $\phi$:
triangular strict monotonicity gives that $\phi$ is a $C^{s}$ diffeomorphism of $(0,1)^d$,
the convex interpolation $\phi_t=(1-t)\mathsf{id}+t\phi$ stays within triangular diffeomorphisms of $(0,1)^d$,
and extending the corresponding velocity to $\mathbb R^d$ yields a global $\mathcal D^s$-flow
whose time-$1$ map restricts to $\Phi$ on $(0,1)^d$.
Thus $\phi\in \mathcal D_0^s((0,1)^d)$.\qedhere
\end{proof}

\bibliographystyle{alpha}
\bibliography{refs}
	
	\bigskip

\begin{minipage}[t]{.5\textwidth}
{\footnotesize{\bf Borjan Geshkovski}\par
  Laboratoire Jacques-Louis Lions\par
  Inria \& Sorbonne Université\par
  4 Place Jussieu\par
  75005 Paris, France\par
 \par
  e-mail: \href{mailto:borjan@mit.edu}{\textcolor{blue}{\scriptsize borjan.geshkovski@inria.fr}}
  }
\end{minipage}% This must go next to `\end{minipage}`
\begin{minipage}[t]{.5\textwidth}
  {\footnotesize{\bf Domènec Ruiz-Balet}\par
  Departament de Matemàtiques i Informàtica\par
  Universitat de Barcelona\par
  Gran Via de les Corts Catalanes 585\par
  08007 Barcelona, Spain\par
 \par
  e-mail: \href{mailto:blank}{\textcolor{blue}{\scriptsize domenec.ruizibalet@ub.edu}}
  }
\end{minipage}

\end{document}